\documentclass{amsart}

\usepackage[all]{xy}
\usepackage{url}
\usepackage{amssymb}
\theoremstyle{plain}

\newtheorem{lemma}[subsection]{Lemma}
\newtheorem{corollary}[subsection]{Corollary}

\theoremstyle{definition}
\newtheorem{definition}[subsection]{Definition}

\newtheorem{notation}[subsection]{Notation}
\newtheorem{comment}[subsection]{Comment}

\theoremstyle{remark}
\newtheorem{remark}[subsection]{Remark}

\numberwithin{equation}{subsection}
\numberwithin{equation}{subsection}
\usepackage{euler} 
\normalfont
\usepackage[T1]{fontenc}

\title{Polygons in Quadratically Closed Rings and Properties of $n$-adically Closed Rings}
\author{Rankeya Datta}
\email{Rankeya Datta rankeya@umich.edu}
\begin{document}
\maketitle

\section{Introduction}
\label{introduction}
\noindent
This paper is inspired by \cite{Art}. In \cite{Art}, Michael Artin proves using an argument involving \'etale algebras, that in an absolutely integrally closed ring (definition \ref{def-absolutely-integrally-closed-ring}), the sum of two prime ideals is either prime or the whole ring. The method of proof, however, is not elementary, and a simpler proof is given in  \cite{Hoch-Hun} for the larger class of quadratically closed rings (definition \ref{def-quadratically-closed-ring}). We have shown Hochster-Huneke's argument in \ref{irreducible-intersection-property-quadratically-closed-rings}. Partly motivated by their elementary proof, one of the main goals of this paper was to find a more elementary proof of [\cite{Art}, 1.7(ii)]. The statement of [\cite{Art}, 1.7(ii)] involves the notion of a polygon in a ring (definition \ref{def-polygons}). It states that an absolutely integrally closed ring has no polygons, and the proof in \cite{Art} again uses \'etale algebras. We find a proof of [\cite{Art}, 1.7(ii)] for quadratically closed rings (see \ref{no-polygons-in-quadratically-closed}, \ref{comment}, and \ref{special-case-polygons-quadratically-closed}) without using \'etale ring maps, and since absolutely integrally closed rings are quadratically closed, the result for absolutely integrally closed rings follows. Once we introduce the notion of a $n$-adically closed ring \ref{def-n-adically-closed}, the proof of [\cite{Art}, 1.7(ii)] is then easily generalized to $2n$-adically closed rings \ref{polygons-2n-adically-closed-rings}. We give below a more detailed account of what each section contains.\\

\noindent
Section \ref{introduction} contains the introduction, and section \ref{conventions} establishes the conventions followed in the rest of the paper. In sections \ref{absolute integral closure}, \ref{absolute integral closure and quotients}, and \ref{absolute integral closure and localization}, we introduce the notion of absolutely integrally closed domains and rings, and state and prove a number of their properties, both ring theoretic and scheme theoretic.  In section \ref{Quadratically Closed Rings and $n$-Adically Closed Rings}, we introduce the notion of quadratically closed rings, and more generally of $n$-adically closed rings, and prove ring and scheme theoretic properties related to these two notions. In most of the lemmas concerning scheme theoretic properties, we assume that the scheme is integral.
We extend \ref{irreducible-intersection-property-quadratically-closed-rings} to $2n$-adically closed rings (see \ref{irreducible-intersection-property-of-2n-adically-closed-rings} for this result). The proof of \ref{irreducible-intersection-property-of-2n-adically-closed-rings} uses the fact that in $n$-adically closed rings, the sum of two prime ideals is always a radical ideal, which we prove in \ref{sum-of-prime-deals-is-radical}. It seems difficult to use the technique of proof employed in \ref{irreducible-intersection-property-of-2n-adically-closed-rings} to prove an analogous statement for $n$-adically closed rings. However, with an additional hypothesis, we are able to prove an analogue of \ref{irreducible-intersection-property-of-2n-adically-closed-rings} for $n$-adically closed rings satisfying this additional hypothesis \ref{irreducible-intersection-property-of-n-adically-closed-rings-with-additional-property}. It is worth observing that in rings satisfying just this additional hypothesis, the sum of two prime ideals is always primary \ref{conditionally-n-adically-closed}. We also give a simpler proof of [\cite{Art}, 1.8] (see \ref{intersections-of-primes}, and \ref{joins-of-2n-adically-closed-rings}), and find that \ref{joins-of-2n-adically-closed-rings} holds for rings which satisfy, what we have termed, the irreducible intersection property (definition \ref{irreducible-intersection-property}). The proof of [\cite{Art}, 1.8] depends on the proof of [\cite{Art}, 1.5(i)]. However, in the context of $2n$-adically closed rings, and indeed, rings satisfying the irreducible intersection property \ref{irreducible-intersection-property}, the proof of [\cite{Art}, 1.5(i)] is less technical (\ref{intersections-of-primes} and \ref{joins-and-iip}), and one can also drop the semi-local hypothesis made in [\cite{Art}, 1.5(i)], if one assumes that $Spec(R)$ is irreducible. In section \ref{irreducible intersection property}, we introduce the notion of the irreducible intersection property, hoping to generalize a few results of section \ref{Quadratically Closed Rings and $n$-Adically Closed Rings}. In \ref{necessary-sufficient-condition-for-iip}, we prove a necessary and sufficient condition for a ring to have the irreducible intersection property. In sections \ref{Quadratically Closed Rings and $n$-Adically Closed Rings} and \ref{irreducible intersection property} we also show that the property of being $n$-adically closed, and the irreducible intersection property, are preserved by taking colimits of directed systems of rings (all having the corresponding property) and ring homomorphisms. Section \ref{semi-local-rings} deals with the process of semi-localization of a ring at prime ideals, and is mainly preparatory work for section \ref{Polygons in Quadratically closed rings}. In section \ref{Polygons in Quadratically closed rings} we introduce the notion of polygons in rings, and give a detailed, but elementary proof of the non-existence of polygons in quadratically closed rings (\ref{no-polygons-in-quadratically-closed}, \ref{comment}, \ref{special-case-polygons-quadratically-closed}), and in general that $2n$-adically closed rings have no polygons \ref{polygons-2n-adically-closed-rings}. Finally, \ref{ack} is the acknowledgements section. 

\section{Conventions}
\label{conventions}
\noindent
In this paper, all rings are assumed to be commutative, with a multiplicatve identity, and a ring homomorphism maps the multiplicative identity of one ring to the multiplicative identity of the other ring. If $\alpha$ is an ideal of a ring $A$, then we denote its radical as $\sqrt{\alpha}$ or $rad(\alpha)$. Given an element $a$ in a ring $A$, if $I \subset A$ is an ideal, then the image of $a$ under the canonical projection $A \rightarrow A/I$ is usually denoted as $\overline{a}$. We denote the nilradical of a ring $A$ by $Nil(A)$. Given a domain $A$, we denote its field of fractions by $Frac(A)$. We also use $\mathbb{N}, \mathbb{Z}, \mathbb{Q}, \mathbb{R}, \mathbb{C}$ to denote the set of natural numbers, integers, rational numbers, real numbers and complex numbers respectively. We do not include $0$ in $\mathbb{N}$.  Given a ring homomorphism $\varphi: A \rightarrow B$, for a prime ideal $q \subset B$, we often denote the contraction $\varphi^{-1}(q)$ as $q \cap A$. We mostly cite Artin's paper \cite{Art}, and when the citation is given followed by a number such [\cite{Art}, 1.7(ii)], it refers to the second part of Definition/ Proposition/ Corollary/ Remark 1.7. Occasionally, we also cite the book \cite{Liu}. As an example, in section \ref{Quadratically Closed Rings and $n$-Adically Closed Rings}, lemma \ref{$n$-adically-closed-normal-schemes}, we cite [\cite{Liu}, 2.4.18(a)]. This stands for chapter 2, proposition 4.18(a) in \cite{Liu}. Whenever we cite something from \cite{stacks-project}, we give the appropriate tag number.

\section{Absolute Integral Closure}
\label{absolute integral closure}

\begin{definition} 
\label{def-absolute-integrally-closed}
An integral domain $A$ is \textbf{absolutely integrally closed} if $A$ is integrally closed in its field of fractions $K$, and $K$ is an algebraically closed field.
\end{definition}

\noindent
It is clear from the definition above that any monic polynomial in $A[x]$ splits into linear factors over $A$. The converse of this statement is also true, and will be proved shortly.

\begin{lemma} 
\label{necessary-sufficient-for-absolutely-integrally-closed}
Let $A$ be an domain. Then $A$ is absolutely integrally closed if and only if  $A$ does not have a proper integral extension which is a domain.
\end{lemma}

\begin{proof}
The proof of $\Rightarrow$ is clear from \ref{def-absolute-integrally-closed}. For the proof of $\Leftarrow$, let $K = Frac(A)$, and $\overline{K}$ the algebraic closure of $K$. Let $\overline{A}$ be the integral closure of $A$ in $\overline{K}$. Then by hypothesis, $A = \overline{A}$. Since $\overline{K} = Frac(\overline{A})$, it follows that $\overline{K} = K$. We also know that $\overline{A}$ is integrally closed in $\overline{K}$, because any element of $\overline{K}$ which is integral over $\overline{A}$ is also integral over $A$. Then $A = \overline{A}$ and $K = \overline{K}$ implies that $A$ is integrally closed in $K = \overline{K}$, completing the proof. \end{proof}

\begin{lemma} 
\label{absolutely-integrally-closed-splitting-monic-polynomial}
A domain $A$ is absolutely integrally closed if and only if any monic polynomial in $A[x]$ splits into linear factors in $A[x]$.
\end{lemma}

\begin{proof}
$\Rightarrow$ is clear from the definition of an absolutely integrally closed domain. For $\Leftarrow$, note that if any monic polynomial in $A[x]$ splits into linear factors in $A[x]$, then $A$ clearly has no proper integral extension which is a domain. Then the result follows by \ref{necessary-sufficient-for-absolutely-integrally-closed}
. 
\end{proof}

\noindent
During the course of the proof of \ref{necessary-sufficient-for-absolutely-integrally-closed} we have implicitly stated how to obtain absolutely integrally closed domains from any given domain. This is made explicit in the next lemma.

\begin{lemma} 
\label{create-absolutely-integrally-closed-from-any-domain}
Let $A$ be a domain, $K$ its fraction field, $\overline{K}$ the algebraic closure of $K$. If $\overline{A}$ denotes the integral closure of $A$ in $\overline{K}$, then $\overline{A}$ is an absolutely integrally closed domain.
\end{lemma}

\begin{proof}
It is clear that $\overline{A}$ is integrally closed in $\overline{K}$. It suffices to show that $Frac(\overline{A}) = \overline{K}$. Let $x \in \overline{K}$. Then there exists a monic polynomial $p(T) \in K[T]$ such that $p(x) = 0$. By clearing denominators, we get that there exists $a \in A - \{0\}$ such that $ax$ is integral over $A$. Thus, $ax \in \overline{A}$, and $x = ax/a \in Frac(\overline{A})$.
\end{proof}

\begin{lemma} 
\label{when-subring-is-absolutely-integrally-closed}
Let $A$ be an absolutely integrally closed domain. If $B \subset A$ is a subring of $A$ which is integrally closed in $A$, then $B$ is an absolutely integrally closed domain.
\end{lemma}

\begin{proof}
Let $f(x) \in B[x]$ be a monic polynomial of degree $n$. Since $f(x) \in A[x]$, it follows that in $A[x]$, $f(x) = (x - \alpha_0)...(x - \alpha_n)$, for $\alpha_i \in A$. Then each $\alpha_i$ is integral over $B$. By hypothesis $B$ is integrally closed in $A$, so that the $\alpha_i$ must be in $B$. We are then done by \ref{absolutely-integrally-closed-splitting-monic-polynomial}. 
\end{proof}

\begin{lemma} 
\label{finite-product-absolutely-integrally-closed}
Let $A_1, ..., A_n$ be absolutely integrally closed domains. Then any monic polynomial in  $A_1 \times ... \times A_n[X]$ splits into linear factors over $A_1 \times .... \times A_n$.
\end{lemma}

\begin{proof}
By induction on $n$, it suffices to prove this for $n = 2$. So, let $f(X) = X^n + (a_{n-1},b_{n-1})X^{n-1} + ... + (a_0,b_0) \in A_1 \times A_2[X]$, be a monic polynomial (where $a_i \in A_1$ and $b_j \in A_2$). Since $A_1$ is integrally closed, we have $X^n + a_{n-1}X^{n-1} + ... + a_0 = (X-\alpha_0)...(X- \alpha_n)$, for $\alpha_i \in A_1$. Similarly, $X^n + b_{n-1}X^{n-1} + ... +b_0 = (X- \beta_0)...(X- \beta_n)$, for $\beta_i \in A_2$. It is then easily seen that $f(X) = (X - (\alpha_0, \beta_0))...(X- (\alpha_n, \beta_n))$, and so we are done. 
\end{proof}

\noindent
Motivated by the previous lemma, we generalize the definition of absolutely integrally closed domains to a larger class of rings.

\begin{definition} 
\label{def-absolutely-integrally-closed-ring}
A ring $R$ is \textbf{absolutely integrally closed} if $R \cong A_1 \times ... \times A_n$, where the $A_i$ are absolutely integrally closed domains.
\end{definition}

\begin{remark} 
\label{rem-observation-absolutely-integrally-closed}
In particular, any absolutely integrally closed domain is an absolutely integrally closed ring. Moreover, note that any absolutely integrally closed ring is reduced, being a finite product of reduced rings.
\end{remark}

\section{Absolute Integral Closure And Quotients}
\label{absolute integral closure and quotients}

\begin{lemma} 
\label{quotient-absolutely-integrally-closed-domain-by-prime}
Let $A$ be an absolutely integrally closed domain, and $p \subset A$ a prime ideal. Then $A/p$ is an absolutely integrally closed domain.
\end{lemma}

\begin{proof}
We will denote the image of $a \in A$ under the natural projection map $A \rightarrow A/p$ as $\overline{a}$. Let $f(x) = x^n + \overline{a_{n-1}}x^{n-1} + ... + \overline{a_0} \in A/p[x]$ be a monic polynomial. Since $A$ is integrally closed, $x^n + {a_{n-1}}x^{n-1} + ... +{a_0} = (x - \alpha_0)...(x - \alpha_n)$ for some $\alpha_i \in A$. Then $f(x) = (x - \overline{\alpha_0})...(x - \overline{\alpha_n})$. We win by \ref{absolutely-integrally-closed-splitting-monic-polynomial}. 
\end{proof}

\begin{lemma} 
\label{quotient-absolutely-integrally-closed-ring-by-prime}
Let $R$ be an absolutely integrally closed ring. If $p \subset R$ is a prime ideal, then $R/p$ is an absolutely integrally closed domain.
\end{lemma}

\begin{proof} 
Without loss of generality, let $R = A_1 \times ... \times A_n$, where the $A_i$ are absolutely integrally closed domains. Then a prime $p$ is of the form $A_1 \times A_2 \times ... \times p_i \times... \times A_n$, where $p_i \in Spec(A_i)$ for some $i$. Then $R/p \cong A_i/p_i$, and the latter ring is an absolutely integrally closed domain by \ref{quotient-absolutely-integrally-closed-domain-by-prime}.
\end{proof}

\begin{lemma} 
\label{image-absolutely-integrally-closed-ring}
Let $R$ be an absolutely integrally closed ring, $D$ a domain, and $\varphi: R \rightarrow D$ a ring map. Then $im(\varphi)$ is an absolutely integrally closed domain.
\end{lemma}

\begin{proof}
Since $D$ is a domain, $(0)$ is a prime ideal. Thus, $ker(\varphi)$ is a prime ideal. Now $R/ker(\varphi)$ is an absolutely integrally closed domain by \ref{quotient-absolutely-integrally-closed-ring-by-prime}. Since $im(\varphi) \cong R/ker(\varphi)$, we are done.
\end{proof}

\begin{lemma} 
\label{integral-ring-map-absolutely-integrally-closed-ring-surjective}
Let $R$ be an absolutely integrally closed ring, $D$ a domain. If $\varphi: R \rightarrow D$ is an integral ring map, then $\varphi$ is surjective, and $D$ is absolutely integrally closed.
\end{lemma}

\begin{proof}
By \ref{image-absolutely-integrally-closed-ring}, $im(\varphi)$ is an absolutely integrally closed domain. Since $\varphi$ is integral, $D$ is integral over $im(\varphi)$. By \ref{necessary-sufficient-for-absolutely-integrally-closed}, $im(\varphi) = D$. This proves both the statements. \end{proof}

\begin{remark} 
\label{rem-finite-ring-maps-of-absolutely-integrally-closed-rings}
In particular, any injective, integral homomorphism of absolutely integrally closed domains is an isomorphism. Since finite ring maps are integral, \ref{integral-ring-map-absolutely-integrally-closed-ring-surjective} is also true for finite ring maps. Another consequence of the previous lemma is the following:
\end{remark}

\begin{lemma} 
\label{integral-morphisms-to-absolutely-integrally-closed-schemes-are-closed-immersions}
Let $Y$ be a scheme such that for all affine open $Spec(A)$ of $Y$, $A$ is an absolutely integrally closed ring. Let $f: X \rightarrow Y$ be an integral morphism of schemes such that $X$ is an integral scheme. Then $f$ is a closed immersion.
\end{lemma}

\begin{proof}
This follows from the following characterization of closed immersions and \ref{integral-ring-map-absolutely-integrally-closed-ring-surjective}: A morphism of schemes $f: X \rightarrow Y$ is a closed immersion if and only if it is affine, and for all affine open $Spec(A) \subset Y$, and $Spec(B) = f^{-1}(Spec(A))$, the ring homomorphism $A \rightarrow B$ is surjective.
\end{proof}

\section{Absolute Integral Closure And Localization}
\label{absolute integral closure and localization}

\noindent

\begin{lemma} 
\label{integral-closure-preserved-localization}
Let $A \subset B$ be an extension of rings, and $C$ the integral closure of $A$ in $B$. If $S \subset A$ is a multiplicative subset, then $S^{-1}C$ is the integral closure of $S^{-1}A$ in $S^{-1}B$.
\end{lemma}

\begin{proof}
See \cite[Lemma 0307]{stacks-project}
\end{proof}

\begin{lemma} 
\label{localization-integral-domains}
Let $A$ be a domain. Then for any multiplicative subset $S$ of $A$, $S^{-1}A$ is a subring of $Frac(A)$.
\end{lemma}

\begin{proof}
If $S$ contains $0$, then $S^{-1}A$ is the zero ring, which is clearly a subring of $Frac(A)$. If $0 \notin S$, then the natural map $\pi: S^{-1}A \rightarrow Frac(A)$, where $\pi(a/s) = a/s$ is an injection, because $A$ is a domain. Hence, $S^{-1}A$ can be identified as a subring of $Frac(A)$ is a natural way.
\end{proof}

\begin{lemma} 
\label{normality-preserved-localization}
Let $A$ be a normal domain. Then for any multiplicative subset $S \subset A$ such that $0 \notin S$, $S^{-1}A$ is integrally closed in $Frac(A)$. Moreover, $Frac(A) = Frac(S^{-1}A)$.
\end{lemma}

\begin{proof}
That $S^{-1}A$ is integrally closed in $Frac(A)$ follows from \ref{integral-closure-preserved-localization}. By \ref{localization-integral-domains}, $S^{-1}A$ is a subring of $Frac(A)$. So $Frac(S^{-1}A) \subset Frac(A)$. On the other hand, $A$ is a subring of $S^{-1}A$. Thus $Frac(A) \subset Frac(S^{-1}A)$.
\end{proof}

\begin{lemma} 
\label{absolutely-integrally-closed-domains-preserved-localization}
Let $A$ be an absolutely integrally closed domain. Then for any multiplicative subset $S$ of $A$ such that $0 \notin S$, $S^{-1}A$ is an absolutely integrally closed domain.
\end{lemma}

\begin{proof}
This is a straightforward application of \ref{normality-preserved-localization}.
\end{proof}

\begin{lemma} 
\label{localization-products}
Let $A_1, ..., A_n$ be domains, and $S \subset A_1 \times ... \times A_n$ a multiplicative subset. If $\pi_i:A_1 \times ... \times A_n \rightarrow A_i$ denote the canonical projection homomorphisms, and $S_i = \pi_i(S)$, then $S^{-1}(A_1 \times \cdots \times A_n) \cong S^{-1}_1A \times ... \times S^{-1}_nA_n$.
\end{lemma}

\begin{proof}
Omitted.
\end{proof}

\begin{lemma} 
\label{absolutely-integrally-closed-ring-preserved-localization}
Let $R$ be an absolutely integrally closed ring. If $S \subset R$ is a multiplicative subset such that $0 \notin S$, then $S^{-1}R$ is absolutely integrally closed.
\end{lemma}

\begin{proof}
We can assume without loss of generality that $R = A_1 \times ... \times A_n$, where the $A_i$ are absolutely integrally closed domains. Since $0 \notin S$, there exists $i \in \{1,...,n\}$ such that $0 \notin S_i$ (here $S_i$ is as in \ref{localization-products}). After rearranging the $A_i's$ and renumbering, we can assume without loss of generality that $\{1, ..., m\} \subset \{1,..., n\}$ is the set of of elements such that $0 \in S_{k}$, for $k = 1, ..., m$, and $0 \notin S_{k}$, for $k \in \{m+1, \cdots, n\}$ (note that $m \neq n$ by an earlier remark). Then by \ref{localization-products}, we have $S^{-1}R \cong S^{-1}_{m+1}A_{m+1} \times ... \times S^{-1}_nA_n$, which is a finite product of absolutely integrally closed domains by \ref{absolutely-integrally-closed-domains-preserved-localization}. This proves the Lemma.
\end{proof}

\begin{lemma} 
\label{normality-is-local}
Let $A$ be an integral domain. Then the following are equivalent:

\begin{enumerate}
\item[(a)] $A$ is integrally closed.
\item[(b)] For all primes ideals $p$, $A_p$ is integrally closed.
\item[(c)] For all maximal ideals $m$, $A_m$ is integrally closed.
\end{enumerate}
\end{lemma}

\begin{proof}
It is clear that (a) $\Rightarrow$ (b) $\Rightarrow$ (c). So, it suffices to prove (c) $\Rightarrow$ (a). Note that all the rings can be naturally identified as subrings of $Frac(A)$, and since $A \subset A_m \subset Frac(A)$, hence $Frac(A) = Frac(A_m)$. Let $\frac{a}{b} \in Frac(A)$. Define the ideal $\mathfrak{I} = \{s \in A: s\frac{a}{b} \in A\}$. If $\frac{a}{b}$ is integral over $A$, it suffices to show that $\mathfrak{I} = (1)$, from which it follows that $\frac{a}{b} \in A$. Let $\frac{a}{b}$ be a root of $X^n + a_{n-1}X^{n-1} +...+a_0$. Since $A \subset A_m$, hence for all maximal ideals $m \subset A$, $\frac{a}{b}$ which is an element of $Frac(A_m)$, is integral over $A_m$. Since $A_m$ is integrally closed we have that $\frac{a}{b} \in A_m$. Thus, there exists $\frac{a_m}{s_m} \in A_m$ such that $\frac{a}{b} = \frac{a_m}{s_m}$. Then for all $m$, $s_m \in \mathfrak{I}$. Since $s_m \in A_m$, $\mathfrak{I}$ is not contained in any maximal ideal of $A$, and as a result $\mathfrak{I} = A$. 
\end{proof}

\begin{lemma} 
\label{absolutely-integrally-closed-domain-is-local}
Let $A$ be an integral domain. Then the following are equivalent:

\begin{enumerate}
\item[(a)] $A$ is absolutely integrally closed.
\item[(b)] For all primes $p$ ideals, $A_p$ is absolutely integrally closed.
\item[(c)] For all maximal ideals $m$, $A_m$ is absolutely integrally closed.
\end{enumerate}
\end{lemma}

\begin{proof}
(a) $\Rightarrow$ (b) follows from \ref{absolutely-integrally-closed-domains-preserved-localization}, and (b) $\Rightarrow$ (c) is obvious. As in \ref{normality-is-local}, it suffices to prove that (c) $\Rightarrow$ (a). By \ref{normality-is-local}, $A$ is integrally closed. Since $A$ is a nonzero ring by hypothesis, by Zorn's lemma, $A$ has at least one maximal ideal say $m$. Since $A$ is a domain, we have $A \subset A_m \subset Frac(A)$. But, $Frac(A) = Frac(A_m)$, and the latter is an algebraically closed field by the hypothesis of (c). Then $A$ is absolutely integrally closed by \ref{def-absolute-integrally-closed}.
\end{proof}

\begin{lemma} 
\label{intersections-absolutely-integrally-closed-domains}
Let $K$ be a field. Let $A$, $B$ be subrings of $K$ which are absolutely integrally closed. Then $A \cap B$ is an absolutely integrally closed domain. In fact, this lemma is true for arbitrary intersections of subrings of $K$ which are absolutely integrally closed.
\end{lemma}

\begin{proof}
Let $C = A \cap B$. Let $f(x) \in C[x]$ be a monic polynomial. Then $f(x) \in A[x], B[x]$. Thus $f(x)$ has $n$ roots in $A$, and $n$ roots in $B$, because $A$ and $B$ are absolutely integrally closed. But, $f(x)$ has at most $n$ roots. Hence, the roots of $f(x)$ in $A$ must be in $B$, and vice versa. This means that the roots of $f(x)$ are actually in $C$. As a result, $f(x)$ factors into a product of linear polynomials over $C$. We are then done by \ref{absolutely-integrally-closed-splitting-monic-polynomial}. It is clear that the method of proof extends to arbitrary intersections of subrings of $K$ which are absolutely integrally closed.
\end{proof}

\begin{lemma} 
\label{affine-locality-absolutely-integrally-closed-domains}
Let $A$ be an integral domain. Then:

\begin{enumerate}
\item[(a)] $A$ is absolutely integrally closed $\Rightarrow$ for all $f \in A - 0$, $A_f$ is absolutely integrally closed.
\item[(b)] If $(f_1, ... ,f_n) = A$, and $A_{f_i}$ is an absolutely integrally closed domain for all $i$, then $A$ is an absolutely integrally closed domain.
\end{enumerate}
\end{lemma}

\begin{proof}
(a) follows from \ref{absolutely-integrally-closed-domains-preserved-localization}. For (b), we have for all $i \in \{1,...,n\}$, $A_{f_i} \subset Frac(A) = Frac(A_{f_i})$, and since the $A_{f_i}$ are absolutely integrally closed, $Frac(A)$ is algebraically closed. By \ref{intersections-absolutely-integrally-closed-domains}, it suffices to show that $A = A_{f_1} \cap ... \cap A_{f_n}$.

\noindent
Let $\frac{a}{t} \in A_{f_1} \cap ... \cap A_{f_n}$. Let $\mathfrak{I} = \{s \in A: s\frac{a}{t} \in A\}$. Then $\mathfrak{I}$ is clearly an ideal of $A$. Since $\frac{a}{t} \in A_{f_i}$, it follows that there exists an $f^{m_i}_i \in \mathfrak{I}$, for some $m_i \in \mathbb{N}$, for all $i \in \{1,...,n\}$. Hence $(f^{m_1}_1,...,f^{m_n}_n) \subset \mathfrak{I}$, which implies that $rad((f^{m_1}_1,...,f^{m_n}_n)) \subset rad(\mathfrak{I})$. But $rad((f^{m_1}_1,...,f^{m_n}_n)) = rad((f_1,...,f_n)) = A$. This shows that $rad(\mathfrak{I}) = A$, which means that $1 \in \mathfrak{I}$, and $\frac{a}{t} \in A$. We have shown that $A_{f_1} \cap ... \cap A_{f_n} \subset A$. Since $A \subset A_{f_i}$ for all $i$, we win.
\end{proof}

\begin{lemma} 
\label{absolutely-integrally-closed-integral-scheme-is-affine-local}
Let $X$ be an integral scheme. Then for every nonempty affine open $Spec(A) \subset X$, $A$ is an absolutely integrally closed domain if and only if there exists an affine open cover $\{Spec(A_i): i \in I\}$ of $X$ such that for all $i \in I$, $A_i$ is an absolutely integrally closed domain.
\end{lemma}

\begin{proof}
The implication $\Rightarrow$ is clear. Now suppose $\{Spec(A_i): i \in I \}$ is an affine open cover of $X$ such that $A_i$ is an absolutely integrally closed domain for each $i \in I$. Let $Spec(A)$ be a nonempty affine open subset of $X$. Since $Spec(A)$ is quasicompact, there exists $i_1,...,i_n \in I$ such that $Spec(A) = (Spec(A) \cap Spec(A_{i_1})) \cup ... \cup (Spec(A) \cap Spec(A_{i_n}))$, where each $Spec(A) \cap Spec(A_{i_j}) \neq \emptyset$. Now, for all $j \in \{1,...,n\}$, $Spec(A) \cap Spec(A_{i_j})$ is a union of nonempty distinguished open subschemes, that are subsequently distinguished open subschemes of $Spec(A)$ and $Spec(A_{i_j})$. By \ref{affine-locality-absolutely-integrally-closed-domains}(a), the ring of global sections of each of these distinguished open subschemes is an absolutely integrally closed domain (because the distinguished open subschemes are distinguished opens of $Spec(A_{i_j})$). Again, since $Spec(A)$ is quasicompact, it can be expressed as a finite union of these distinguished open subschemes, say $Spec(A) = D({f_1}) \cup ... \cup D(f_m)$, where each $f_k \in A$ and $A_{f_k}$ is an absolutely integrally closed domain. Since $A$ is a domain, by \ref{affine-locality-absolutely-integrally-closed-domains}(b), $A$ is absolutely integrally closed.
\end{proof}

\begin{lemma} 
\label{absolutely-integrally-closed-scheme}
Let $X$ be an integral scheme. Then the following are equivalent:
\begin{enumerate}
\item[(a)] For all open $U \subset X$, $\mathcal{O}_X(U)$ is an absolutely integrally closed domain.
\item[(b)] For all affine open $Spec(A) \subset X$, $\mathcal{O}_X(Spec(A)) = A$ is an absolutely integrally closed domain.
\end{enumerate}
\end{lemma}

\begin{proof}
That (a) $\Rightarrow$ (b) is clear. Hence, assume (b). Let $\eta$ be the generic point of $X$. For all $x \in X$, by choosing an affine open neighborhood $Spec(A_x)$ of $x$, it is clear that the stalk $\mathcal{O}_{X,x}$ is an absolutely integrally closed domain, with fraction field isomorphic to $\mathcal{O}_{X,\eta}$, which is thus an algebraically closed field. Then by [\cite{Liu}, 2.4.18], for each $x \in X$ and open $U \subset X$, $\mathcal{O}_{X,x}$ and $\mathcal{O}_X(U)$ can be identified as a subrings of $\mathcal{O}_{X,\eta}$. After this identification is made, one can see that $\mathcal{O}_X(U) = \bigcap_{x \in U} \mathcal{O}_{X,x}$ [\cite{Liu}, 2.4.18(c)]. It follows by \ref{intersections-absolutely-integrally-closed-domains} that $\mathcal{O}_X(U)$ is an absolutely integrally closed domain.
\end{proof}

\begin{lemma} 
Let $R$ be an absolutely integrally closed ring, $D$ a domain, and $\varphi: R \rightarrow D$ a ring map. If $S \subset R$ is a multiplicative subset such that $S \cap \ker(\varphi) = \emptyset$, then $im(\varphi^*)$ is an absolutely integrally closed domain, where $\varphi^*:S^{-1}R \rightarrow S^{-1}D$ is the induced map.
\end{lemma}

\begin{proof}
By \ref{absolutely-integrally-closed-domains-preserved-localization} and \ref{absolutely-integrally-closed-ring-preserved-localization}, $S^{-1}R$, $S^{-1}D$ are absolutely integrally closed. We have $ker(\varphi^{*}) = S^{-1}(ker(\varphi))$, and the latter is a prime ideal in $S^{-1}R$ because $S \cap ker(\varphi) = \emptyset$. Now apply \ref{image-absolutely-integrally-closed-ring}.
\end{proof}

\begin{lemma} 
Let $R$ be an integrally closed ring, $D$ a domain, and $\varphi:R \rightarrow D$ an integral ring map. If $S \subset R$ is a multiplicative subset such that $S \cap ker(\varphi) = \emptyset$, the $\varphi^*:S^{-1}R \rightarrow S^{-1}D$ is surjective, and $S^{-1}D$ is absolutely integrally closed.
\end{lemma}

\begin{proof}
The ring map $\varphi^*:S^{-1}R \rightarrow S^{-1}D$ is integral. Now apply \ref{integral-ring-map-absolutely-integrally-closed-ring-surjective}.
\end{proof}

\section{Quadratically Closed Rings and $n$-Adically Closed Rings}
\label{Quadratically Closed Rings and $n$-Adically Closed Rings}

\begin{definition} 
\label{def-quadratically-closed-ring}
A ring $R$ is \textbf{quadratically closed} if every monic quadratic polynomial in $R[x]$ has a root in $R$.
\end{definition}

\noindent
A simple consequence of this definition is that a ring $R$ is quadratically closed if and only if every monic quadratic polynomial splits into linear factors in $R[x]$. We will not bother to write this as a separate lemma. Note also that any absolutely integrally closed ring is quadratically closed. 

\begin{lemma} 
\label{quadratically-closed-rings-and-quotients}
Any quotient of a quadratically closed ring is quadratically closed. As a result, the homomorphic image of a quadratically closed ring is quadratically closed.
\end{lemma}

\begin{proof} 
The proof of the first statement is similar to that of \ref{quotient-absolutely-integrally-closed-domain-by-prime} and we omit it. For the second statement, note that the image of a quadratically closed ring under a ring map is isomorphic to a quotient of the quadratically closed ring, hence is quadratically closed by the first statement.
\end{proof}

\begin{lemma} 
\label{2-is-a-unit}
Let $R$ be a quadratically closed ring such that $2 \in R^{\times}$. Then $R$ is quadratically closed if and only if $R^2 = R$.
\end{lemma}

\begin{proof}
$\Rightarrow$ For all $r \in R$, the polynomial $x^2 - r$ has a root over $R$. This shows that $R^2 = R$.\

\noindent
$\Leftarrow$ Let $x^2 + bx + c \in R[x]$. Then $x^2 + bx + c = (x + {b}/{2})^2 - ({b^2 -4c})/{4} = (x + {b}/{2})^2 - a^2$, where $a^2 =  ({b^2 -4c})/{4}$, for some $a \in R$, which exists because $R^2 = R$. Thus, $x^2 + bx + c = (x + {b}/{2} + a)(x + {b}/{2} - a)$, and we are done.
\end{proof}

\begin{lemma} 
\label{localization-quadratically-closed}
Let $R$ be a quadratically closed ring, and $S \subset R$ a multiplicative subset. Then $S^{-1}R$ is quadratically closed.
\end{lemma}

\begin{proof}
Let $x^2 + ({b_1}/{s_1})x + {b_0}/{s_0} \in S^{-1}R[x]$. Then consider the polynomial $x^2 + b_1s_0x + b_0s^2_1s_0 \in R[x]$. This has a root $r \in R$. Hence, $r^2 + b_1s_0r + b_0s^2_1s_0 = 0$, which implies that  ${r^2+ b_1s_0r + b_0s^2_1s_0}/{(s_1s_0)^2} = 0$ in $S^{-1}R$. From this last equality we get that ${r}/{s_1s_0}$ is a root of $x^2 + ({b_1}/{s_1})x + {b_0}/{s_0}$. 
\end{proof}

\begin{lemma} 
\label{degree-2-extension-of-fractions-fields}
Let $A$ be a quadratically closed domain. Then $Frac(A)$ has no degree $2$ field extensions.
\end{lemma}

\begin{proof}
By \ref{localization-quadratically-closed}, $Frac(A)$, being a localization of $A$, is quadratically closed. Hence the result follows.
\end{proof}

\begin{lemma} 
\label{fibers-of-maps-to-quadratically-closed-rings}
Let $A \rightarrow B$ be a ring map such that $B$ is quadratically closed. Then the ring of global functions of a fiber of $Spec(B) \rightarrow Spec(A)$ over a point of $Spec(A)$ is quadratically closed.
\end{lemma}

\begin{proof}
Let $p$ be a point of $Spec(A)$. Then we know that the fiber over $p$ is $Spec(B) \times_A Spec(\kappa(p)) = Spec(B \otimes_A \kappa(p)) = Spec(S^{-1}(B/pB))$, where $S = (A/p - 0)^{-1}$. The ring $S^{-1}(B/pB)$ is quadratically closed by \ref{quadratically-closed-rings-and-quotients} and \ref{localization-quadratically-closed}.
\end{proof}

\begin{remark} 
\label{rem-quadratically-closed-domains}
Since any monic quadratic polynomial in $A[x]$ with a root in $A$, splits completely into linear factors in $A$, one can easily deduce that \ref{intersections-absolutely-integrally-closed-domains} is true by replacing absolutely integrally closed domains by quadratically closed domains. As a result, \ref{affine-locality-absolutely-integrally-closed-domains} is true for absolutely integrally closed replaced by quadratically closed. Hence, we obtain analogues of \ref{absolutely-integrally-closed-integral-scheme-is-affine-local} and \ref{absolutely-integrally-closed-scheme} for integral schemes whose affine open subschemes are the spectra of quadratically closed domains:
\end{remark}

\begin{lemma} 
\label{quadratically-closed-schemes}
Let $X$ be an integral scheme. Then the following properties are equivalent:
\begin{enumerate}
\item[(a)] For all open $U \subset X$, $\mathcal{O}_X(U)$ is a quadratically closed domain.
\item[(b)] For every nonempty affine open $Spec(A) \subset X$, $A$ is quadratically closed domain. 
\item[(c)] There exists an affine open cover $\{Spec(A_i): i \in I\}$ of $X$ such that for all $i \in I$, $A_i$ is a quadratically closed domain.
\end{enumerate}
\end{lemma}

\noindent
For a ring $A$, it is not usually the case that the intersection of two irreducible closed subsets of $Spec(A)$ is an irreducible closed subset. However, for rings that are quadratically closed, we obtain the following surprising result whose proof is given in [\cite{Hoch-Hun}, 9.2]:

\begin{lemma} 
\label{irreducible-intersection-property-quadratically-closed-rings}
Let $R$ be a quadratically closed ring, and $p_1, p_2 \subset R$ be two prime ideals (not necessarily distinct). Then either $p_1 + p_2 = R$ or $p_1 + p_2$ is prime. Geometrically this means that two irreducible closed subsets of $Spec(R)$ are either disjoint, or their intersection is an irreducible closed subset.
\end{lemma}

\begin{proof}
If $p_1 + p_2 = R$, then we are done. So assume $p_1 + p_2 \neq R$. Let $x, y \in R$ such that $xy \in R$. Let $z = y - x$. Then $x^2 + zx = xy = a + b$ where $a \in p_1$ and $b \in p_2$. Consider the polynomial $T^2 + zT - a \in R[T]$. Since $R$ is quadratically closed, this polynomial has a root, say $u$. Thus, $u^2 + zu = a \in p_1$. Since $p_1$ is prime, this means that $u \in p_1$ or $u + z \in p_1$. Now, $b = x^2 + zx - a = x^2 + zx - u^2 - zu = (x - u)(x + u + z) \in p_2$. Thus, $x - u \in p_2$ or $x + u + z \in p_2$. Then we have
\begin{equation}
x = (x - u) + u = (x + u + z) - (u + z)
\end{equation}
\begin{equation}
y = (x+ u + z) - u = (x - u) + (u + z)
\end{equation}   

\noindent
Hence we get $x \in p_1 + p_2$ or $y \in p_1 + p_2$.
\end{proof}

\noindent
In fact, we can prove a more general version of the previous lemma. But first, we need a definition.

\begin{definition} 
\label{def-n-adically-closed}
For $n \in \mathbb{N}$, a ring $R$ is \textbf{$n$-adically closed} if every monic polynomial of degree $n$ with coefficients in $R$ has a root.
\end{definition}

\noindent
Thus, a $2$-adically closed ring is just a quadratically closed ring.

\begin{remark} 
\label{rem-explanation-of-previous-definition}
Note that the above definition does not imply that any monic polynomial of degree $n$ can be completely factored into a product of linear terms. It just says that every monic polynomial of degree $n$ is guaranteed to have at least one root. Also note that the definition is not redundant in the sense that there are $n$-adically closed rings, for $n \geq 3$, which are not quadratically closed. For example, $\mathbb{R}$ is $n$-adically closed for every odd $n$, but is not quadratically closed.
\end{remark}

\begin{remark} 
\label{subrings-of-n-adically-closed-rings}
Subrings of $n$-adically closed rings are not $n$-adically closed. Again, $\mathbb{R}$ is $n$-adically closed for every odd $n$, but $\mathbb{Q}$ is not $n$-adically closed for any $n$.
\end{remark}

\begin{lemma} 
\label{quotients-localizations-products-of-n-adically-closed-rings}
The property of a ring being $n$-adically closed is preserved under quotients, localizations and arbitrary products.
\end{lemma}

\begin{proof}
That the property of a ring being $n$-adically closed is preserved under quotients and arbitrary products is quite clear. Let us show that this property is preserved under localization. So, we let $R$ be an $n$-adically closed ring, and $S \subset R$ be a multiplicative subet. Let $X^n + (r_{n-1}/s_{n-1})X^{n-1} + \cdots + (r_1/s_1)X + r_0/s_0$ be a monic polynomial of degree $n$ with coefficients in $S^{-1}R$. For $i = 1, \cdots, n-1$, let $\beta_{n-i} = r_{n-i}s^{i-1}_{n-i}(\prod_{j \neq n-i}s^{i}_j)$. Observe that\
\begin{equation}
\label{equation}
\beta_{n-i}/(s_{n-1}\cdots s_0)^n = (r_{n-i}/s_{n-i})(1/(s_{n-1}\cdots s_0)^{n-i}). 
\end{equation}
Now the monic polynomial $X^n + \beta_{n-i}X^{n-1} + \cdots + \beta_1X + \beta_0$ has a root $\alpha$ in $R$. Using \ref{equation}, it is then easy to see that $\alpha/s_{n-1}\cdots s_0$ is a root of $X^n + (r_{n-1}/s_{n-1})X^{n-1} + \cdots + (r_1/s_1)X + r_0/s_0$.
\end{proof}

\begin{lemma} 
\label{when-a-subring-of-n-adically-closed-ring-is-n-adically-closed}
Let $R$ be an $n$-adically closed ring, and $A$ a subring of $R$ such that $A$ is integrally closed in $R$. Then $R$ is $n$-adically closed.
\end{lemma}

\begin{proof}
Let $f(x) \in A[x]$ be a monic polynomial of degree $n$. Then $f(x) \in R[x]$, and so $f(x)$ has a root $r \in R$. Then $r$ is integral over $A$, and since $A$ is integrally closed in $R$, we must have $r \in A$.
\end{proof}

\begin{lemma} 
\label{$n$-adically-closed-normal-schemes}
Let $X$ be an irreducible, normal scheme, with generic point $\eta$. Then the following are equivalent:
\begin{enumerate}
\item[(i)] $\mathcal{O}_{X,\eta}$ is $n$-adically closed.
\item[(ii)] For all $x \in X$, $\mathcal{O}_{X,x}$ is $n$-adically closed.
\item[(iii)] For all open $U \subset X$, $\mathcal{O}_X(U)$ is $n$-adically closed.
\item[(iv)] For all affine open subsets $Spec(A) \subset X$, $A$ is $n$-adically closed.
\item[(v)] There exists an affine open cover $\{Spec(A_i)\}_{i\in I}$ of $X$ such that for all $i \in I$, $A_i$ is $n$-adically closed.
\end{enumerate}
\end{lemma}

\begin{proof}
It is clear that (iii) $\Rightarrow$ (iv) $\Rightarrow$ (v). Now assume (i). Since $X$ is normal, it is reduced, hence integral (because $X$ is irreducible by hypothesis). Let $x \in X$, and $Spec(A)$ be an affine open subset containing $x = [p]$. Then it is easy to see that $Frac(A) = \mathcal{O}_{X,\eta}$ and $\mathcal{O}_{X,x} = A_p$. Thus, $Frac(\mathcal{O}_{X,x}) = \mathcal{O}_{X,\eta}$. Since $\mathcal{O}_{X,x}$ is a normal domain by hypothesis and $\mathcal{O}_{X,\eta}$ is $n$-adically closed, it follows by \ref{when-a-subring-of-n-adically-closed-ring-is-n-adically-closed} that $\mathcal{O}_{X,x}$ is $n$-adically closed. Thus, (i) $\Rightarrow$ (ii). Now assume (ii). Then $\mathcal{O}_{X,\eta}$ is $n$-adically closed. If $U = \emptyset$, then (iii) is trivially true. Suppose that $U \neq \emptyset$. Since $X$ is an integral scheme, we have $\mathcal{O}_X(U) = \bigcap_{x \in U} \mathcal{O}_{X,x}$. We also know that $Frac(\mathcal{O}_X(U)) = \mathcal{O}_{X,\eta}$ [\cite{Liu}, 2.4.18(a)], and that for all $x \in U$, we have natural inclusions $\mathcal{O}_X(U) \hookrightarrow \mathcal{O}_{X,x} \hookrightarrow \mathcal{O}_{X,\eta}$. Hence, if $\alpha \in \mathcal{O}_{X, \eta}$ is integral over $\mathcal{O}_X(U)$, then for all $x \in U$, $\alpha$ is integral over $\mathcal{O}_{X,x}$. Since each $\mathcal{O}_{X,x}$ is a normal domain by hypothesis, it follows that $\alpha \in \mathcal{O}_{X,x}$ for all $x \in U$. In particular, $\alpha \in \bigcap_{x \in U} \mathcal{O}_{X,x} = \mathcal{O}_X(U)$. This shows that $\mathcal{O}_X(U)$ is a normal domain. It follows by \ref{when-a-subring-of-n-adically-closed-ring-is-n-adically-closed} that $\mathcal{O}_X(U)$ is $n$-adically closed when $U \neq \emptyset$. This establishes that (ii) $\Rightarrow$ (iii). It suffices to prove that (v) $\Rightarrow$ (i). Well, it is clear that $\mathcal{O}_{X,\eta} = Frac(A_i)$ for any $i$, and the latter is $n$-adically closed by \ref{quotients-localizations-products-of-n-adically-closed-rings}.
\end{proof}

\begin{lemma} 
\label{sufficient-condition-for-radical-ideal}
Let $R$ be a ring, and $I$ an ideal of $R$. Suppose there exists a natural number $n$ greater than 1 such that for all $x \in I$, if $x^n \in I$ then $x \in I$. Then $I$ is a radical ideal.
\end{lemma} 

\begin{proof}
Assume $\sqrt{I} \neq I$. Let $x \in \sqrt{I} - I$. Then there exists a least $m \in \mathbb{N}$ such that $x^m \in I$. If $n|m$, then writing $m = nk$ for some $k \in \mathbb{N}$, we get by hypothesis that $x^k \in I$. But $k < m$, which contradicts our choice of $m$. Now suppose that $n$ does not divide $m$. Let $m = qn + r$ for $q, r \in \mathbb{N}$ such that $0 < r < n$. Then $x^mx^{n-r} = x^{(q+1)n} \in I$. By hypothesis we get $x^{q+1} \in I$. Since $n > 1$, we have $q+1 < m$, which again contradicts our choice of $m$. By contradiction it follows that $\sqrt{I} - I = \emptyset$, proving that $\sqrt{I} = I$.
\end{proof}

\begin{lemma} 
\label{arithmetic-fact}
Let $n \in \mathbb{N}$. Then there exists $m \in \mathbb{N}$ such $(n-1)^{m} < n^{m-1}$.
\end{lemma}

\begin{proof}
We have $(n-1)^m/n^{m-1} = (n-1)(1 - 1/n)^{m-1}$. Since $1 - 1/n < 1$, it follows that lim$_{m \rightarrow \infty} (1 - 1/n)^{m-1} = 0$. Thus, for $\epsilon = 1/n$, there exists $M \in \mathbb{N}$ such that for all $m \geq M$, $(1 - 1/n)^{m-1} < 1/n$. Then, $(n-1)(1 - 1/n)^{m-1} < (n-1)/n < 1$. We have proved that there exists $m \in \mathbb{N}$ such that $(n-1)(1 -1/n)^{m-1} < 1$. This implies that $(n-1)^m < n^{m-1}$.
\end{proof}

\begin{lemma} 
\label{sum-of-prime-deals-is-radical}
Let $R$ be a $n$-adically closed ring, for $n \in \mathbb{N}$. If $p, q$ are prime ideals of $R$, then $p + q$ is a radical ideal.
\end{lemma}

\begin{proof}
We will show that for all $x \in p + q$, if $x^{n} \in p + q$ then $x \in p + q$. We will then be done by \ref{sufficient-condition-for-radical-ideal}. Hence let $x^{n} \in p + q$. Then there exists $a \in p, b \in q$ such that $x^{n} = a + b$. Since $R$ is $n$-adically closed, there exists $t_1 \in R$ such that $t^{n}_1 = a$. Since $a \in p$ and $p$ is prime, this means that $t_1 \in p$, and in particular, $t^n_1 \in p$. Now, $b = x^{n} - a = x^{n} - t^{n}_1 = (x - t_1)(x^{n-1} + x^{n-1}t_1 + \cdots + t^{n-1}_1) \in q$. Thus, $x - t_1 \in q$ or $x^{n-1} + x^{n-1}t_1 + \cdots + t^{n-1}_1 \in q$. Since $t_1 \in p$, we get  $x \in p + q$ or $x^{n-1} \in p+q \Rightarrow x^{n-1} \in p+q$. So we have proved the following: if $x^{n} \in p + q$, then $x^{n-1} \in p+q$. By \ref{arithmetic-fact}, there exists $m \in \mathbb{N}$ such that $(n-1)^m < n^{m-1}$. For $i = 2, \cdots, m-1$, there exists $t_i$ such that $t_{i-1} = t^n_i$.\\

\noindent
Then we have the following:

\medskip \noindent
$x^{n-1} = t^{n(n-1)}_1 \in p+q$, and so, $t^{(n-1)^2}_1 \in p+q$.

\medskip \noindent
$t^{(n-1)^2}_1 = t^{n(n-1)^2}_2 \in p+q$, and so, $t^{(n-1)^3}_2 \in p+q$.

\medskip\noindent
$t^{(n-1)^3}_2 = t^{n(n-1)^3}_3 \in p+q$, and so, $t^{(n-1)^4}_3 \in p+q$.

\medskip\noindent
$\cdots$\

\medskip\noindent
$t^{(n-1)^{m-1}}_{m-2} = t^{n(n-1)^{m-1}}_{m-1} \in p+q$, and so, $t^{(n-1)^m}_{m-1} \in p+q$.

\medskip\noindent
Let $k = n^{m-1} - (n-1)^m$. Then $x = t^n_1 = t^{n^2}_2 = t^{n^3}_3 = \cdots = t^{n^{m-1}}_{m-1} = (t^{(n-1)^m}_{m-1})t^{k}_{m-1} \in p+q$.
\end{proof}

\noindent
We are now in a position to generalize the result of \ref{irreducible-intersection-property-quadratically-closed-rings}.

\begin{lemma} 
\label{irreducible-intersection-property-of-2n-adically-closed-rings}
Let $R$ be a $2n$-adically closed ring, for $n \in \mathbb{N}$. If $p, q$ are prime ideals of $R$, then either $p + q = R$, or $p+q$ is a prime ideal.
\end{lemma}

\begin{proof}
Suppose $p+q \neq R$. It suffices to show that $p+q$ is prime. Let $x, y \in R$ such that $xy \in p+q$. Let $z = y - x^n$. Then $x^{2n} + zx^n = x^ny \in p+q$. Thus, there exist $a \in p, b \in q$ such that $x^{2n} + zx^n = a + b$. Since $R$ is $2n$-adically closed, there exists $u \in R$ such that $u^{2n} + zu^n - a = 0$. This means that $u^{2n} + zu^n = a \in p$, and since $p$ is prime, we get $u^n \in p$ or $u^n + z \in p$. We also have $b = x^{2n} + zx^n - a = x^{2n} + zx^n - u^{2n} - zu^n = (x^n - u^n)(x^n + u^n + z)$. Since $b \in q$, we get $x^n - u^n \in q$ or $x^n + u^n + z \in q$. We then have
\begin{equation}
x^n = (x^n - u^n) + u^n = (x^n + u^n + z) - (u^n + z)
\end{equation}
\begin{equation}
y = x^n + z = (x^n + u^n + z) - u^n = (x^n - u^n) + (u^n + z)
\end{equation}
As a result, it follows that $x^n \in p+q$ or $y \in p+q$. Since $p+q$ is radical by \ref{sum-of-prime-deals-is-radical}, we get $x \in p+q$ or $y \in p+q$.
\end{proof}

\medskip\noindent
Generalizing \ref{irreducible-intersection-property-of-2n-adically-closed-rings} to $n$-adically closed rings by similar methods appears to be difficult. However, with an additional hypothesis, we get the following result:

\begin{lemma} 
\label{irreducible-intersection-property-of-n-adically-closed-rings-with-additional-property}
Let $R$ be an $n$-adically closed ring, for $n > 1$. Assume that for any prime ideal $p \subset R$, and for all $a, b \in R$ such that $a \notin p$ and $b \in p$, the polynomial $T^n +aT^{n-1} + b \in R[T]$ has a root $\alpha \in R - p$. Then the sum of two prime ideals is either prime or the whole ring.
\end{lemma}

\begin{proof}
Let $P, Q$ be two prime ideals of $R$ such that $P + Q \neq R$.  Let $x, y \in R$ such that $xy \in P+Q$. Let $z = y - x$.  Then either $z$ is an element of  $P$, or $z \notin P$.
If $z \in P$, from the equation $xz = xy - x^2$, it follows that $x^2 \in P+Q$. Applying \ref{sum-of-prime-deals-is-radical} we get $x \in P+Q$. Hence, assume that $z \notin P$. Then, $x^n + zx^{n-1} = x^{n-1}y \in P + Q$. In particular, there exists $c \in P, d \in Q$ such that $x^{n-1}y = c + d$. By hypothesis, the polynomial $T^n + zT^{n-1} - c$ has a root $u \notin P$. Since $u^n + zu^{n-1} = c \in P$, it follows that $u + z \in P$. On the other hand, $d = x^n + zx^{n-1} - u^n - zu^{n-1}  = \big{(}x - u\big{)}\big{(}x^{n-1} + x^{n-2}u + \cdots + u^{n-1} + z(x^{n-2} + x^{n-3}u + \cdots + u^{n-2})\big{)} \in Q$. Thus, $x - u \in Q$ or $x^{n-1} + x^{n-2}u + \cdots + u^{n-1} + z(x^{n-2} + x^{n-3}u + \cdots + u^{n-2}) \in Q$. If $x-u \in Q$, then $y = x + z = (x-u) + (u+z) \in P + Q$. If $x^{n-1} + x^{n-2}u + \cdots + u^{n-1} + z(x^{n-2} + x^{n-3}u + \cdots + u^{n-2}) \in Q$, then $\big{(}x^{n-1} + x^{n-2}u + \cdots + u^{n-1} + z(x^{n-2} + x^{n-3}u + \cdots + u^{n-2})\big{)} - \big{(}u+z\big{)}\big{(}x^{n-2} + x^{n-3}u + \cdots + u^{n-2}\big{)} = x^{n-1} \in P+Q$. Again by \ref{sum-of-prime-deals-is-radical}, we win!
\end{proof}

\begin{remark}
\label{conditionally-n-adically-closed}
Let $R$ be a ring. If there exists $n > 1$ such that for any prime ideal $p \subset R$, and for all $a, b \in R$ such that $a \notin p$ and $b \in p$, the polynomial $T^n +aT^{n-1} + b \in R[T]$ has a root $\alpha \in R - p$, then the proof of \ref{irreducible-intersection-property-of-n-adically-closed-rings-with-additional-property} shows that in this ring, the sum of two prime ideals will always be a primary ideal.
\end{remark}

\begin{lemma} 
\label{decreasing-degree}
Let $R$ be an $n$-adically closed reduced ring. Then for all $m \in \mathbb{N}$ such that $m|n$, $R$ is $m$-adically closed.
\end{lemma}

\begin{proof} 
Let $f(T) \in R[T]$ be a monic polynomial of degree $m$. Then $f(T)^d$ is a monic polynomial of degree $n$, where $d = n/m$. By hypothesis, $f(T)^d$ has a root $\alpha \in R$. Thus, $f(\alpha)^d = 0$, and since $R$ is reduced, we get $f(\alpha) = 0$.
\end{proof}

\begin{lemma} 
\label{intersections-of-primes}
Let $R$ be a $2n$-adically closed ring such that $Spec(R)$ is irreducible. Let $p_1, \cdots, p_n$ be prime ideals of $R$. Then there is a unique prime ideal $P$ which is maximal among those prime ideals contained in $p_1 \cap \cdots \cap p_n$.
\end{lemma}

\begin{proof}
Let $\Sigma$ be the set of prime ideals of $R$ which are contained in $p_1\cap \cdots \cap p_n$. Since $Spec(R)$ is irreducible, it follows that $Nil(R) \in \Sigma$, where $Nil(R)$ is the nilradical of $R$. Thus $\Sigma$ is non-empty. Then by Zorn's lemma, $\Sigma$ has maximal elements. Suppose that $\Sigma$ does not have a unique maximal element, and let $P_1, P_2$ be distinct maximal elements of $\Sigma$. Then $P_1, P_2 \subsetneq P_1 + P_2$, and $P_1 + P_2 \subset p_1 \cap \cdots \cap p_n \neq R$. Thus by \ref{irreducible-intersection-property-of-2n-adically-closed-rings}, $P_1 + P_2$ is a prime ideal. But then $P_1 + P_2 \in \Sigma$, contradicting the maximality of $P_1$ and $P_2$.
\end{proof}

\begin{comment} 
\label{comment-on-Artin-1.5(i)}
Note that one can follow the pattern of argument in \ref{intersections-of-primes} to prove [\cite{Art}, 1.5(i)], without the assumption that $R$ admits no finite, \'etale algebra $R'$ of rank $2$, which splits at $p$ and $p'$ as in [\cite{Art}, 1.5(i)], keeping the other hypotheses of [\cite{Art}, 1.5(i)].
\end{comment}

\begin{definition} 
\label{def-joins-of-rings}
Let $K$ be a ring, and $A, B$ be subrings of $K$. We define the \textbf{join} of $A$ and $B$, denoted $[A,B]$ to be the smallest subring of $K$ (smallest with respect to inclusion) that contains $A$ and $B$.
\end{definition}

\begin{lemma} 
\label{explicit-computation-of-joins-of-rings}
Let $A, B$ be subrings of a ring $K$. Consider the set $\Gamma = \{ a_1b_1 + \cdots + a_nb_n : a_i \in A, b_j \in B\}$. This is a subring of $K$, and $\Gamma = [A,B]$.
\end{lemma}

\begin{proof}
It is clear that $\Gamma$ is a subring of $K$. It is also clear that any subring of $K$ that contains $A, B$ also contains elements of the form $a_1b_1 + \cdots + a_nb_n$, where $a_i \in A, b_j \in B$. Thus, any subring of $K$ that contains $A, B$ contains $\Gamma$. In particular, this means that $\Gamma$ is a subring of $[A, B]$. On the other hand, $\Gamma$ clearly contains $A, B$. By the definition of the join of $A$ and $B$, it follows that $\Gamma = [A, B]$.
\end{proof}

\begin{lemma} 
\label{localization-and-joins}
Let $A$ be a domain, and $S, T$ be multiplicative subsets of $A$. If we consider $S^{-1}A, T^{-1}A$ to be subrings of $Frac(A)$, then $[S^{-1}A, T^{-1}A] = (ST)^{-1}A$.
\end{lemma}

\begin{proof}
This is clear from \ref{explicit-computation-of-joins-of-rings}.
\end{proof}

\begin{lemma} 
\label{joins-of-2n-adically-closed-rings}
Let $A$ be a $2n$-adically closed domain, and $p, q$ be prime ideals of $A$. If $A_p, A_q$ are considered to be subrings of $Frac(A)$, then the join $[A_p, A_q]$ is a local ring.
\end{lemma}

\begin{proof}
Let $S = A - p$, and $T = A - q$. Then by \ref{localization-and-joins} it follows that $[A_p, A_q] = (ST)^{-1}A$. By \ref{intersections-of-primes} let $P$ be the unique prime ideal, maximal among those contained in $p \cap q$. Note that $ST \cap P = \emptyset$, for if not, then there exists $st \in ST$, where $s \in S, t \in T$ such that $st \in P$. Then $s \in P$ or $t \in P$, which implies that $s \in p$ or $t \in q$, a contradiction. Thus, $P((ST)^{-1}A)$ is a prime ideal of $(ST)^{-1}A$. If $Q$ is a prime ideal of $(ST)^{-1}A$, then $Q = \overline{q}((ST)^{-1}A)$, where $\overline{q}$ is a prime ideal of $A$ not intersecting $ST$. Since $S, T \subset ST$, it follows that $\overline{q} \cap S = \emptyset = \overline{q} \cap T$. Thus, $\overline{q} \subset p, q$. By the maximality of $P$, it follows that $\overline{q} \subset P$. Hence, $Q \subset P((ST)^{-1}A)$. This shows that $P((ST)^{-1}A)$ is the unique maximal ideal of $(ST)^{-1}A$, proving that $[A_p,A_q] = (ST)^{-1}A$ is a local ring.
\end{proof}

\begin{remark} 
\label{rem-relation-Artin}
Since absolutely integrally closed rings are also $2n$-adically closed, it follows that \ref{joins-of-2n-adically-closed-rings} is true for absolutely integrally closed domains. This is observed in [\cite{Art}, 1.8]. But the proof given above is much more elementary, and avoids the usage of semi-local rings and \'etale algebras.
\end{remark}

\begin{lemma} 
\label{colimits-directed-n-adically-closed-rings}
Let $(I, \leq)$ be a directed partially ordered set, that is, $I$ is a partially ordered set such that for all $i, j \in I$, there exists $k \in I$ with $i, j \leq k$. Let $(R_i, \varphi_{ij})_{i,j \in I}$ be a directed system of rings and ring homomorphisms indexed by $I$, that is, for all $i \leq j$ we have a unique ring map $\varphi_{ij}: R_i \rightarrow R_j$ satisfying:
\begin{enumerate}
\item[(i)] $\varphi_{ii} = id_{R_i}$.
\item[(ii)] $\varphi_{jk} \circ \varphi_{ij} = \varphi_{ik}$, for $i \leq j \leq k$.
\end{enumerate}
If each $R_i$ is an $n$-adically closed ring, then $colim_{i \in I} R_i$ is an $n$-adically closed ring.
\end{lemma}

\begin{proof}
Let us first do the case when each $R_i$ is $n$-adically closed. For the proof of this lemma, we will need to look at the actual construction of $colim_{i \in I} R_i$. Note that the underlying set of $colim_{i \in I}R_i$ is $(\bigsqcup_{i \in I}R_i)/{\sim}$, where for $i, j \in I$, $r_i \in R_i, r_j \in R_j$, $(r_i,i) \sim (r_j,j)$ if and only if there exists $k \in I$ such that $i,j \leq k$ and $\varphi_{ik}(r_i) = \varphi_{jk}(r_j)$. Then denoting the equivalence class of $(r_i,i)$ as $[r_i,i]$, we define addition and multiplication as follows: $[r_i,i] + [r_j,j] = [\varphi_{ik}(r_i) + \varphi_{jk}(r_j), k]$ and $[r_i, i][r_j,j] = [\varphi_{ik}(r_i)\varphi_{jk}(r_j),k]$ for $i,j \leq k$ (note such a $k$ exists because $I$ is directed). Now let $x^n + [r_{i_{n-1}}, i_{n-1}]x^{n-1} + \cdots + [r_{i_1},i_1]x + [r_{i_0},i_0] \in colim_{i \in I}R_i[x]$ be an polynomial of degree $n$. Since $I$ is directed, there exists $k \in I$ such that $i_0, i_1, \cdots, i_{n-1} \leq k$. Then $x^n + [r_{i_{n-1}}, i_{n-1}]x^{n-1} + \cdots + [r_{i_1},i_1]x + [r_{i_0},i_0] =x^n + [\varphi_{i_{n-1}k}(r_{i_{n-1}}), k]x^{n-1} + \cdots + [\varphi_{i_1k}(r_{i_1}),k]x + [\varphi_{i_0k}(r_{i_0}),k]$. The polynomial $x^n + \varphi_{i_{n-1}k}(r_{i_{n-1}})x^{n-1} + \cdots + \varphi_{i_0k}(r_0)$ has a root $\alpha$ in $R_k$. Then by the definition of addition and multiplication above, it is then easy to see that $[\alpha,k]$ is a root of $x^n + [r_{i_{n-1}}, i_{n-1}]x^{n-1} + \cdots + [r_{i_1},i_1]x + [r_{i_0},i_0] $.
\end{proof}

\begin{corollary} 
\label{Stalks-of-n-adically-closed-rings}
Let $(X, \mathcal{O}_X)$ be a ringed space. If for all open $U \subset X$, $\mathcal{O}_X(U)$ is $n$-adically closed, then for all $x \in X$, the stalk $\mathcal{O}_{X,x}$ is $n$-adically closed. 
\end{corollary}

\begin{proof}
Since the stalk $\mathcal{O}_{X,x}$ is the colimit of a directed system of rings and ring homomorphisms, where all the rings are $n$-adically closed, it follows by \ref{colimits-directed-n-adically-closed-rings} that $\mathcal{O}_{X,x}$ is $n$-adically closed.
\end{proof}

\begin{remark} 
\label{rem-converse-of-{Stalks-of-n-adically-closed-rings}}
Although arbitrary products of $n$-adically closed rings are $n$-adically closed, since subrings of $n$-adically closed rings need not be $n$-adically closed, one cannot immediately deduce the converse of \ref{Stalks-of-n-adically-closed-rings}, namely that if all the stalks of a ringed space $(X, \mathcal{O}_X)$ are $n$-adically closed, then for all open $U \subset X$, $\mathcal{O}_X(U)$ is $n$-adically closed. However, we have seen \ref{$n$-adically-closed-normal-schemes} that if $X$ is a normal, irreducible scheme, then the converse of \ref{Stalks-of-n-adically-closed-rings} does hold.
\end{remark}

\section{Irreducible intersection property}
\label{irreducible intersection property}

\noindent 
In \ref{irreducible-intersection-property-of-2n-adically-closed-rings} we proved that $2n$-adically closed rings have the property that the intersection of two irreducible closed subsets of their spectra is either empty or is irreducible. We will now prove a necessary and sufficient condition for a ring to have this property. In fact, as will be clear from the statement of the next lemma, we already used a weaker version this fact in the proof of \ref{intersections-of-primes}. But first, we make a provisional definition. Note that this definition is not standard in the literature.

\begin{definition} 
\label{irreducible-intersection-property}
Let $R$ be a ring. Then $R$ satisfies the \textbf{irreducible intersection property} if for any prime ideals $p_1, p_2 \subset R$, either $p_1 + p_2 = R$, or $p_1 + p_2$ is prime.
\end{definition}

\noindent
Examples of rings that satisfy the irreducible intersection property are fields, Artinian rings (because all prime ideals are maximal), Dedekind domains, $2n$-adically closed rings \ref{irreducible-intersection-property-of-2n-adically-closed-rings}, absolutely integrally closed rings, valuation rings (since prime ideals in a valuation ring are totally ordered by inclusion).

\begin{lemma} 
\label{necessary-sufficient-condition-for-iip}
Let $R$ be a ring. Then the following properties are equivalent:
\begin{enumerate}
\item[(a)] $R$ satisfies the irreducible intersection property.
\item[(b)] For all ideals $I \subsetneq R$, if $\Sigma = \{p \in Spec(R): p \subset I\} \neq \emptyset$, then $\Sigma$ has a unique maximal element with respect to inclusion.
\item[(c)] If $p_1, p_2$ are distinct prime ideals of $R$ such that $p_1 + p_2 \neq R$, then $\{p \in Spec(R) : p \subset p_1 + p_2\}$ has a unique maximal element with respect to inclusion.
\end{enumerate}
\end{lemma}

\begin{proof}
(a) $\Rightarrow$ (b): Suppose $I \subsetneq R$ is an ideal such that $\Sigma = \{p \in Spec(R): p \subset I\} \neq \emptyset$. Then $\Sigma$ has maximal elements by Zorn's Lemma. Assume for contradiction that $\Sigma$ has more than one maximal element. Then let $p_1, p_2$ be two distinct maximal elements of $\Sigma$. Then $p_1 + p_2 \subset I$, so by (a), $p_1 + p_2$ is prime. Hence, $p_1 + p_2 \in \Sigma$. But $p_1, p_2 \subsetneq p_1 + p_2$, contradicting the maximality of $p_1, p_2$.

\medskip\noindent
(b) $\Rightarrow$ (c): Trivial.

\medskip\noindent
(c) $\Rightarrow$ (a): Let $p_1, p_2$ be prime ideals of $R$. Suppose $p_1 + p_2 \neq R$. Then it suffices to show that $p_1 + p_2$ is prime. Let $\Sigma = \{p \in Spec(R): p \subset p_1 + p_2\}$. By (c), $\Sigma$ has a unique maximal element with respect to inclusion. Let this maximal element be $P$. Let $\Gamma = \{p \in Spec(R): p_1 \subset p \subset p_1 + p_2\}$. Then $\Gamma \neq \emptyset$, because $p_1 \in \Gamma$. By Zorn's Lemma, $\Gamma$ has maximal elements with respect to inclusion. But any maximal element of $\Gamma$ is also a maximal element of $\Sigma$. This proves that $p_1 \subset P$. Similarly, $p_2 \subset P$. Thus, $p_1 + p_2 \subset P$. By our choice of $P$, we get $p_1 + p_2 = P$. Since $P$ is prime, we are done.
\end{proof}

\begin{lemma} 
\label{constructions-preserving-iip}
The irreducible intersection property is preserved under quotients, localization, and finite products.
\end{lemma}

\begin{proof}
Let $R$ be a ring that satisfies the irreducible intersection property. Let $I$ be an ideal of $R$. Every prime ideal of $R/I$ is of the form $p(R/I)$, where $p$ is a prime ideal of $R$ that contains $I$. Let $p_1(R/I), p_2(R/I)$ be two prime ideals of $R/I$, and suppose that $p_1(R/I) + p_2(R/I) = (p_1 + p_2)R/I \neq R/I$. Then $p_1 + p_2 \neq R$. Thus, $p_1 + p_2$ is prime, which means that $p_1(R/I) + p_2(R/I)$ is prime. This proves that $R/I$ satisfies the irreducible intersection property.

\medskip\noindent
Let $S \subset R$ be a multiplicative subset. Then every prime ideal of $S^{-1}R$ is of the form $p(S^{-1}R)$, where $p$ is a prime ideal of $R$ such that $p \cap S = \emptyset$. If $p_1(S^{-1}R), p_2(S^{-1}R)$ are two prime ideals of $S^{-1}R$, we again have $p_1(S^{-1}R) + p_2(S^{-1}R) = (p_1 + p_2)S^{-1}R$. Then the proof that $S^{-1}R$ satisfies the irreducible intersection property is identical to the one above for quotients.

\medskip\noindent
Let $R_1, \cdots, R_n$ be rings that satisfy the irreducible intersection property. Then we need to show that $R_1 \times \cdots \times R_n$ satisfies the irreducible intersection property. By induction on $n$, we reduce to the case where $n = 2$. Note that prime ideals of $R_1 \times R_2$ are of the form $p \times R_2$, where $p$ is a prime ideal of $R_1$, or $R_1 \times q$ where $q$ is a prime ideal of $R_2$. It is easy to see that $p \times R_2 + R_1 \times q = R_1 \times R_2$. Hence, given prime ideals $p_1 \times R_2, p_2 \times R_2, R_1 \times q_1, R_1 \times q_2$, it suffices to show that $p_1 \times R_2 + p_2 \times R_2$ is prime or equal to $R_1 \times R_2$, and $R_1 \times q_1 + R_2 \times q_2$ is prime or equal to $R_1 \times R_2$. But this follows from the observations that $p_1 \times R_2 + p_2 \times R_2 = (p_1 + p_2) \times R_2$, $R_1 \times q_1 + R_1 \times q_2 = R_1 \times (q_1 + q_2)$, and the fact that $R_1, R_2$ satisfy the irreducible intersection property.
\end{proof}

\begin{remark} 
\label{rem-iip}
The above lemma shows that the image, under a ring map, of a ring satisfying the irreducible intersection property also satisfies the irreducible intersection property.
\end{remark}

\begin{lemma} 
\label{fibers-of-maps-to-rings-with-iip}
Let $A \rightarrow B$ be a ring map such that $B$ satisfies the irreducible intersection property. Let $p \subset A$ be a prime ideal. Then $B \otimes_A \kappa(p)$ satisfies the irreducible intersection property, where $\kappa(p) = Frac(A/p)$. Hence, the fibers of $A \rightarrow B$ satisfy the irreducible intersection property.
\end{lemma}

\begin{proof}
We have $B \otimes_A \kappa(p) \cong (A/p - 0)^{-1}(B/pB)$, and the latter ring satisfies the irreducible intersection property by \ref{constructions-preserving-iip}.
\end{proof}

\begin{lemma} 
\label{joins-and-iip}
Let $R$ be a domain that the satisfies the irreducible intersection property. Let $p, q$ be prime ideals of $R$. Then the join $[R_p, R_q]$ in $Frac(R)$ is a local ring.
\end{lemma}

\begin{proof}
By \ref{necessary-sufficient-condition-for-iip}, $\{P \in Spec(R) : P \subset p \cap q\}$ has a unique maximal element with respect to inclusion. The rest of the proof is then exactly the same as \ref{joins-of-2n-adically-closed-rings}.
\end{proof}

\begin{lemma} 
\label{pushing-forward-iip}
Let $A$ be a ring that satisfies the irreducible intersection property. Let $\varphi: A \hookrightarrow B$ be an injective ring map that satisfies the following properties:
\begin{enumerate}
\item[(i)] $\varphi$ satisfies the Lying-Over property, that is for every prime ideal $p \subset A$, there exists a prime ideal $q \subset B$ such that $q \cap A = p$.
\item[(ii)] For every prime $q \subset B$, $(q \cap A)B = q$, that is the extension of the contraction of $q$ is $q$ itself.
\end{enumerate}
Then $B$ satisfies the irreducible intersection property. (Note that (i) and (ii) imply that there exists a unique prime ideal lying over every prime ideal of A, i.e., the map $Spec(B) \rightarrow Spec(A)$ is a bijection.)
\end{lemma}

\begin{proof} Let $I \subsetneq B$ be an ideal such that $\Sigma = \{q \in Spec(B) : q \subset I\} \neq \emptyset$. By \ref{necessary-sufficient-condition-for-iip} it suffices to show that $\Sigma$ has a unique maximal element with respect to inclusion. If $J = I \cap A$, then $J \neq A$, and $\Gamma = \{p \in Spec(A): p \subset J\} \neq \emptyset$. Since $A$ satisfies the irreducible intersection property, by \ref{necessary-sufficient-condition-for-iip} we get that $\Gamma$ has a unique maximal element with respect to inclusion, say $P$. By (i), there exists a prime ideal $Q \subset B$ such that $P = Q \cap A$. Then by (ii), $Q = PB \subset I$. Thus, $Q \in \Sigma$. Suppose that $Q$ is not a maximal element of $\Sigma$. Then there exists an element $Q' \in \Sigma$ such that $Q \subsetneq Q'$. As a result, we get $P = Q \cap A \subset Q' \cap A$. Since $Q' \cap A \in \Gamma$ and $P$ is the unique maximal element of $\Gamma$, we get $P = Q' \cap A$. By (ii), $Q' = PB = Q$, a contradiction. Thus, $Q$ is a maximal element of $\Sigma$. Let $Q'$ be any maximal element of $\Sigma$. Then $Q' \cap A$ is a maximal element of $\Gamma$, for is not, there exists $P'$ such that $Q' \cap A \subsetneq P'$, and there exists $Q'' \in \Sigma$ such that $Q'' \cap A = P'$. Then $Q' = (Q' \cap A)B \subset P'B = Q''$, and by the maximality of $Q'$, we get $Q' = Q''$, proving that $Q' \cap A = P'$, a contradiction. Since $\Sigma$ has a unique maximal element, we get that any maximal element of $\Gamma$ lies over $P$. Hence every maximal element of $\Gamma$ equals $Q$, completing the proof.
\end{proof}

\begin{lemma} 
\label{pullin-back-iip}
Let $\varphi: A \rightarrow B$ be a ring map satisfying the following conditions:
\begin{enumerate}
\item[(i)] For all prime ideals $p \subset A$, $pB$ is a prime ideal.
\item[(ii)] For all ideals $I \subset A$, $IB \cap A = I$.
\end{enumerate}
If $B$ satisfies the irreducible intersection property, then $A$ satisfies the irreducible intersection property.
\end{lemma}

\begin{proof}
Let $p_1, p_2 \subset A$ be prime ideals. By (i), $p_1B, p_2B$ are prime ideals, and by (ii), $(p_1B + p_2B) \cap A = (p_1 + p_2)B \cap A = p_1 + p_2$. Since either $p_1B + p_2B = B$ or $p_1B + p_2B$ is prime, we get either $p_1 + p_2 = A$ or $p_1 + p_2$ is prime.
\end{proof}

\begin{corollary} 
\label{some-oberservations}
Let $A$ be a ring. 
\begin{enumerate}
\item[(a)] If the polynomial ring $A[x_1,\cdots,x_n]$ in $n$ variables satisfies the irreducible intersection property, then $A$ satisfies the irreducible intersection property.
\item[(b)] If the power series ring $A[[x_1, \cdots, x_n]]$ satisfies the irreducible intersection property, then so does $A$.
\end{enumerate}
\end{corollary}

\begin{proof}
Let $\varphi: A \hookrightarrow A[x_1,\cdots,x_n]$, $\phi: A \hookrightarrow A[[x_1, \cdots, x_n]]$ be the natural inclusions. Then $\varphi, \phi$ satisfies properties (i) and (ii) of \ref{pullin-back-iip}, and we win.\end{proof}

\begin{lemma} 
\label{colimits-and-iip}
Let $I$ be a partially ordered set. Let $(R_i, \varphi_{ij})_{i, j \in I}$ be a directed system of rings and ring homomorphisms, indexed by $I$. If each $R_i$ satisfies the irreducible intersection property, then $colim_{i \in I} R_i$ satisfies the irreducible intersection property.
\end{lemma}

\begin{proof}
For all $i \in I$, let $\varphi_i: R_i \rightarrow colim_{i \in I} R_i$ denote the  projection maps. Let $p, q$ be prime ideals of $colim_{i \in I} R_i$ such that $p + q \neq (1)$. It suffices to show that $p + q$ is prime. Let $[r_i, i], [r_j, j] \in colim_{i \in I} R_i$ such that $[r_i,i][r_j,j] \in p + q$. Then there exists $k \geq i,j$ and $a, b \in R_k$ such that $[a,k] \in p, [b,k] \in q$, and $[\varphi_{ik}(r_i)\varphi_{jk}(r_j), k] = [a,k] + [b,k] = [a + b, k]$. Hence, there exists $k \leq l$ such that $\varphi_{kl}(\varphi_{ik}(r_i)\varphi_{jk}(r_j)) = \varphi_{kl}(a + b) \Rightarrow \varphi_{il}(r_i)\varphi_{jl}(r_j) = \varphi_{kl}(a) + \varphi_{kl}(b)$.

\noindent
 Let $p_l = \varphi^{-1}_l(p)$, and $q_l = \varphi^{-1}_l(q)$. Then $p_l, q_l$ are prime, and moreover, $p_l + q_l \subset \varphi^{-1}_l(p + q) \neq R_l$. Thus, $p_l + q_l$ is prime since $R_l$ satisfies the irreducible intersection property. Since of $[a,k] = [\varphi_{kl}(a), l]$ and $[b,k] = [\varphi_{kl}(b), l]$ it follows that $\varphi_{kl}(a) \in p_l, \varphi_{kl}(b) \in q_l$. As a result we get, $\varphi_{il}(r_i)\varphi_{jl}(r_j) = \varphi_{kl}(a) + \varphi_{kl}(b) \in p_l + q_l$. Then we must have $\varphi_{il}(r_i) \in p_l + q_l$ or $\varphi_{jl}(r_j) \in p_l + q_l$. This means that $[r_i, i] = [\varphi_{il}(r_i), l] \in p + q$,  or $[r_j, j] = [\varphi_{jl}(r_j), l] \in p+q$.
\end{proof}

\noindent
As a consequence of the above lemma, we immediately obtain that,

\begin{corollary} 
\label{stalks-of-schemes-with-iip}
If $(X, \mathcal{O}_X)$ is a ringed space, such that for all open $U \subset X$, $\mathcal{O}_X(U)$ satisfies the irreducible intersection property, then for all $x \in X$, the stalk $\mathcal{O}_{X,x}$ satisfies the irreducible intersection property.
\end{corollary}

\section{Semi-local rings}
\label{semi-local-rings}

\noindent
Recall that a semi-local ring is a ring with finitely many maximal ideals. The next result shows us how to obtain semi-local rings from a given ring.

\begin{lemma} 
\label{semi-localization}
Let $A$ be a ring, and $p_1,\cdots,p_n$ be distinct prime ideals of $A$. Let $S = (A - p_1) \cap \cdots \cap (A - p_n)$. Then $S^{-1}A$ is a semi-local ring. 
\end{lemma}

\begin{proof}
Since $p_i \cap S = \emptyset$ for all $i$, it follows that $p_i(S^{-1}A)$ are prime ideals of $S^{-1}A$. Let $q$ be a prime ideal of $S^{-1}A$ and $p = q \cap A$. Then $q = p(S^{-1}A)$, and $p \cap S = \emptyset$. Thus, $p \subset p_1 \cup \cdots \cup p_n$. Hence by Prime Avoidance it follows that $p \subset p_i$ for some $i$. In particular, this mean that $q \subset p_i(S^{-1}A)$. This shows that the maximal ideals of $S^{-1}A$ are among the $p_i(S^{-1}A)$.
\end{proof}

\begin{remark} 
\label{rem-maximal-ideals-semi-localization}
In the lemma above, if there are no inclusions among the $p_i$, then the $p_i(S^{-1}A)$ are the maximal ideals of $S^{-1}A$.
\end{remark}

\begin{definition} 
\label{def-semi-localization}
Let $A$ be a ring, and $p_1,\cdots,p_n$ be distinct prime ideals of $A$ such that $p_i \nsubseteq p_j$ for all $i \neq j$. By the \textbf{semi-local ring at $(p_1,\cdots,p_n)$} we mean that ring $S^{-1}A$, where $S = (A - p_1) \cap \cdots \cap (A - p_n)$.
\end{definition}

\begin{lemma} 
\label{localization-of-localized-rings-at-primes}
Let $A$ be a ring, $S$ a multiplicative subset, and $p \subset A$ a prime ideal such that $p \cap S = \emptyset$. Then $S^{-1}A_{p(S^{-1}A)} \cong A_p$.
\end{lemma}

\begin{proof}
Omitted.
\end{proof}

\section{Polygons in Quadratically closed rings}
\label{Polygons in Quadratically closed rings}

\noindent
We now introduce the notion of a polygon in a ring. We will give the definition in [\cite{Art}, 1.7(ii)]. Our goal will be to prove that quadratically closed rings do not have polygons. This is already shown in [\cite{Art}, 1.7(ii)]. However, we hope to give a more direct proof, without using \'etale algebras.

\begin{definition} 
\label{def-polygons}
Let $R$ be a ring. A \textbf{polygon} in $R$ consists of the following data:

\begin{enumerate}
\item[(i)] An integer $n > 2$.

\item[(ii)] Irreducible closed subset $C_0, \cdots, C_{n-1}$, and irreducible closed sets $D_0, \cdots, D_{n-1}$ of $Spec(R)$ such that $D_j \subset C_i$ if and only if $i = j$ or $i = j+1$ (mod $n$).
\end{enumerate}
We call $n$ the number of \textbf{sides} of the polygon.
\end{definition}

\begin{notation} 
\label{notation}
Throughout this section, we will denote a polygon in our ring $R$ as $P = (n, C_i, D_j)$, where $n, C_i, D_j$ are as defined in \ref{def-polygons}, i.e., in particular, we will always assume that $D_i \subset C_i, C_{i+1}$. We will sometimes denote a polygon by the letter $P$, omitting $n, C_i, D_j$ when they are clear from context.
\end{notation}

\begin{remark} 
\label{remark-polygons}
One can verify that the geometrical picture of a polygon in a ring is, in most cases, what one would expect a polygon to look like. However, the picture could be more complicated. For instance, in the definition above, it could happen that $C_i \cap C_{i+1}$, which are like "vertices of a polygon", intersect with each other, making the picture more complicated. However, if  $Q_i$ is the generic point of $D_i$, then passing onto the semi-local ring at $(Q_1, \cdots, Q_{n-1})$, one "gets a picture" that is closer to being like the picture of a polygon from Euclidean geometry. Furthermore, if $R$ is quadratically closed, then the situation is even better, as the "vertices" $C_{i} \cap C_{i+1}$ are themselves irreducible closed subsets, and hence connected. This is the content of the next two lemmas.
\end{remark}

\begin{lemma} 
\label{polygons-under-localization}
Let $R$ be a ring, and $P = (n, C_i, D_j)$ be a polygon in $R$. Let $P_i$ be the generic point of $C_i$, and $Q_j$ be the generic point of $D_j$. Then
\begin{enumerate}
\item [(a)] The primes $P_i$ are all distinct, and for all $i \neq k$, $P_i \nsubseteq P_k$. The primes  $Q_j$ are all distinct, and for all $j \neq k$, $Q_j \nsubseteq Q_k$. The primes $P_i$ are distinct from the primes $Q_j$.

\item [(b)] Let $\overline{R}$ be the semi-local ring at $(Q_0, \cdots, Q_{n-1})$. Then the maximal ideals of $\overline{R}$ are $Q_j\overline{R}$. Given the map on Spec, $\varphi: Spec(\overline{R}) \rightarrow Spec(R)$, if $\overline{C_i} = \varphi^{-1}(C_i)$, and $\overline{D_j} = \varphi^{-1}(D_j)$, then we get a polygon $\overline{P} = (n, \overline{C_i}, \overline{D_j})$ in $\overline{R}$.

\item [(c)] In the polygon $\overline{P} = (n, \overline{C_i}, \overline{D_j})$, if $k \neq i, i+1$, then $\overline{C_i} \cap \overline{C_k} = \emptyset$. In particular, for all $j \neq l$, $\overline{D_j} \cap \overline{D_l} = \emptyset$.
\end{enumerate}
\end{lemma}

\begin{proof} (a) We actually have to prove three separate assertions as part of (a). For the first assertion, it suffices to show that for all $i \neq k$, $P_i \nsubseteq P_k$. Assume for contradiction that there exists $i \neq k$, for $i, k \in \{0, \cdots, n-1\}$, such that $P_i \subset P_k$. Then $C_k \subset C_i$. Now, $D_{k-1}, D_k \subset C_k$. Thus, $D_{k-1}, D_k \subset C_i$. Since $n > 2$, this means that $i = k$ by \ref{def-polygons}
(ii). This is a contradiction. 

\noindent
For the second assertion, it again suffices to show that for all $j \neq k$, $Q_i \nsubseteq Q_k$. Assume that there exists $j \neq k \in \{0, \cdots, n-1\}$ such that $Q_j \subset Q_k$. Then $D_k \subset D_j$. Thus, $D_{k} \subset C_k, C_{k+1}, C_{j}, C_{j+1}$. Since $j \neq k$, and $n > 2$, the set $\{k, k+1$(mod $n$),$ j, j+1$(mod $n)\}$ has at least three elements, which contradicts \ref{def-polygons}(ii).

\medskip\noindent
For the final assertion, assume that $P_i = Q_j$ for some $i, j \in \{0, \cdots, n-1\}$, where $i,j$ are not necessarily distinct. Then $D_j = C_i$. In particular, this implies that $D_j \subset C_i$. Thus, $i = j$ or $i = j+1$ (mod $n$). If $i = j$, then $D_{j-1} \subset C_{j} = D_j$. But this means that $Q_j \subset Q_{j-1}$, and we just proved that this cannot happen. If $i = j+1$ (mod $n$), then $D_{j+1} \subset C_{j+1} = D_j$, which is again impossible.

\medskip\noindent
(b) By (a), the $Q_i$ are all distinct, and have no strict inclusions. Then by \ref{rem-maximal-ideals-semi-localization} it follows that $\{Q_i\overline{R}\}$ is the set of maximal ideals of $\overline{R}$. Since $P_i \subset Q_i$, it follows that $P_i\overline{R}$ is a prime ideal of $\overline{R}$ for all $i$. It is easily seen that $\overline{C_i} = V(P_i\overline{R})$ and $\overline{D_i} = V(Q_i\overline{R})$, and so it follows that $\overline{C_i}$, $\overline{D_i}$ are irreducible closed subsets of $\overline{R}$. The ring map $R \rightarrow \overline{R}$ that induces the map on $\varphi$ has the following property: If $S = \bigcap_{i} (R - Q_i)$, and $\alpha$ is a prime of $R$ such that $\alpha \cap S = \emptyset$, then $(R \rightarrow \overline{R})^{-1}(\alpha\overline{R}) = \alpha$. Using this one can show that $\overline{D_i} \subset \overline{C_j}$ if and only if $j = i$ of $j = i+1$ (mod $n$). This establishes that $\overline{P} = (n, \overline{C_i}, \overline{D_j})$ is a polygon in $\overline{R}$.

\medskip \noindent
(c) Suppose $\overline{C_i} \cap \overline{C_k} \neq \emptyset$ for some $k \neq i, i + 1$. (Note that such  $i, k$ can exist because by (b), $\overline{P}$ is a polygon with $n$ sides, where $n > 2$). This means that $P_i\overline{R} + P_k\overline{R} \neq \overline{R}$. So there is a maximal ideal of $\overline{R}$ that contains $P_i\overline{R} + P_k\overline{R}$. By (b) we know that the closed one point sets $\overline{D_j}$ are in bijection with the maximal ideals of $\overline{R}$. Thus, there exists $l$ such that $\overline{D_l} \subset \overline{C_i}  \cap \overline{C_k}$. But, this is impossible by \ref{def-polygons}
(ii), because $\overline{P}$ is a polygon. Since the $\overline{D_i}$ are in bijection with the maximal ideals of $\overline{R}$, it follows that for $j \neq l$, $\overline{D_j} \cap \overline{D_l}  = \emptyset$.
\end{proof}

\begin{lemma} 
\label{polygons-in-quadratically-closed-rings}
Let $R$ be a quadratically closed ring, and $P = (n, C_i, D_j)$ be a polygon in $R$. Then for all $i \in \{0, \cdots, n-1\}$, $C_i \cap C_{i+1}$ is an irreducible closed set (indices mod $n$).
\end{lemma}

\begin{proof}
Note that $D_i \subset C_i \cap C_{i+1}$, for all $i$. If $Q_i$ is the generic point of $D_i$, and $P_i$ the generic point of $C_i$, this means that $P_i + P_{i+1} \subset Q_i$. Hence, $P_i + P_{i+1} \neq R$. Then we are done by \ref{irreducible-intersection-property-quadratically-closed-rings}.
\end{proof}

\begin{lemma} 
\label{exact-sequence}
Let $R$ be a ring, and $I_0, I_1$ be ideals of $R$. Consider the map $\phi: R/I_0 \times R/I_1 \rightarrow R/I_0 + I_1$ defined as follows: $\phi(\overline{a_0}, \overline{a_1}) = \overline{\overline{a_0 - a_1}}$ $($here $\overline{a_k}$ is the image of the element $a_k \in R$ in $R/I_k$, and $\overline{\overline{a_0 - a_1}}$ is the image of $a_0 - a_1$ in $R/I_0 + I_1$$)$. If $\phi(\overline{a_0}, \overline{a_1}) = 0$, then there exists $a \in R$ such that $(\overline{a}, \overline{a}) = (\overline{a_0}, \overline{a_1})$.
\end{lemma}

\begin{proof}
Suppose $\phi(\overline{a_0}, \overline{a_1}) = 0$. Then $a_0 - a_1 \in I_0 + I_1$. Thus, there exist $i_0 \in I_0$ and $i_1 \in I_1$ such that $a_0 - a_1 = i_1 - i_0$. Let $a = a_0 + i_0 = a_1 + i_1$. It is then clear that $(\overline{a}, \overline{a}) = (\overline{a_0}, \overline{a_1})$.
\end{proof}

\begin{lemma} 
\label{generalized-exact-sequence}
Let $R$ be a ring. Let $n \geq 2$. Let $I_0, \cdots, I_{n-1}$ be ideals of $R$ such that for all distinct $i, j, k \in \{0, \cdots, n-1\}$, $I_i + I_j + I_k = R$. Let $M = \big{\{}(i, j): i, j \in \{0,\cdots, n-1\}$ and $i < j\big{\}}$. For $(i,j) \in M$, let $\alpha_{i,j} = I_i + I_j$. Consider the ring homomorphism $\phi: \prod_{0 \leq k \leq n-1} R/I_k \rightarrow \prod_{(i,j) \in M} R/\alpha_{i,j}$ defined as follows: $\phi(\overline{a_0}, \cdots, \overline{a_{n-1}}) = (\overline{\overline{a_i - a_j}})_{(i,j) \in M}$, for $a_0, \cdots, a_{n-1} \in R$. If $\phi(\overline{a_0}, \cdots, \overline{a_{n-1}}) = 0$, then there exists $a \in R$ such that $(\overline{a},\cdots, \overline{a}) = (\overline{a_0}, \cdots, \overline{a_{n-1}})$.
\end{lemma}

\begin{proof}
We will prove this by induction on $n$. The base case $n = 2$ is just \ref{exact-sequence}. Assume the statement holds for $n = m$ for $m \geq 2$. Then, for $n = m +1$, let $\phi(\overline{a_0}, \cdots, \overline{\alpha_{m-1}}, \overline{\alpha_{m}}) = 0$. By the induction hypothesis, there exists $a \in R$ such that $(\overline{a}, \cdots, \overline{a}, \overline{a}) = (\overline{a_0}, \cdots, \overline{a_{m-1}}, \overline{a})$. Then, $(\overline{a_0}, \cdots, \overline{a_{m-1}}, \overline{a_{m}}) - (\overline{a}, \cdots, \overline{a}, \overline{a}) = (0, \cdots, 0, \overline{a_{m} - a})$, and clearly, $\phi(0, \cdots, 0, \overline{a_{m} - a}) = 0$. In particular, this means that for all $i \neq m$, $a_{m} - a \in \alpha_{i,m}$. Thus, $a_{m} - a \in \bigcap_{i \neq m} \alpha_{i,m}$. Now, for all $i, j < m$ such that $i \neq j$, $\alpha_{i,m} + \alpha_{j,m} =  I_i + I_j + I_m = R$ by hypothesis. Thus, $\bigcap_{i \neq m} \alpha_{i,m} = \prod_{i \neq m} \alpha_{i,m} = (\prod_{i \neq m} I_i) + I_m = (\prod_{0 \leq i \leq m-1}I_i) + I_m$. Then there exists $b \in \prod_{0 \leq i \leq m-1}I_i$, $i_m \in I_m$ such that $a_m - a = b + i_m$. It immediately follows that $(\overline{b}, \cdots, \overline{b}, \overline{b}) = (0, \cdots, 0, \overline{a_m - a})$. As a result, $(\overline{a_0}, \cdots, \overline{a_{m-1}}, \overline{a_m}) = (\overline{a + b}, \cdots, \overline{a+b}, \overline{a+b})$. We are then done by induction.
\end{proof}

\begin{corollary} 
\label{exact-sequence-for-proof-of-quadratically-closed}
Under the hypotheses of \ref{generalized-exact-sequence}, let $\varphi: R/I_0 \cap \cdots I_{n-1} \rightarrow \prod_{0 \leq i \leq n-1} R/I_i$ be the canonical inclusion. If $\phi(\overline{a_0}, \cdots, \overline{a_{n-1}}) = 0$, then there exists $\gamma \in R/I_0 \cap \cdots \cap I_{n-1}$ such that $\varphi(\gamma) = (\overline{a_0}, \cdots, \overline{a_{n-1}})$.
\end{corollary}

\begin{proof}
By \ref{generalized-exact-sequence}, there exists $a \in R$ such that $(\overline{a}, \cdots, \overline{a}) = (\overline{a_0}, \cdots, \overline{a_{n-1}})$. Just take $\gamma$ to be the image of $a$ in $R/I_0 \cap \cdots \cap I_{n-1}$.
\end{proof} 

\begin{lemma} 
\label{no-polygons-in-quadratically-closed}
Let $R$ be a quadratically closed ring. There are no polygons in $R$.
\end{lemma}

\begin{comment} 
\label{comment}
Let $R$ be a quadratically closed ring with a polygon $P = (n, C_i, D_j)$. Let $P_i$ be the generic point of $C_i$ and $Q_j$ the generic point of $D_j$. Since the property of being quadratically closed is preserved under localization, by passing to the semi-local ring at $(Q_0, \cdots, Q_{n-1})$, it suffices to show \ref{polygons-under-localization}(c) that a polygon $P = (n, C_i, D_j)$, with the property that $C_i \cap C_k = \emptyset$ for all $k \neq i, i+1$, does not exist in $R$. Since $R$ is a quadratically closed ring, by \ref{polygons-in-quadratically-closed-rings} we further reduce to the following situation:
\end{comment}

\begin{lemma} 
\label{special-case-polygons-quadratically-closed}
Let $R$ be a quadratically closed ring. Then a polygon of the type $P = (n, C_i, C_j \cap C_{j+1})$, where for all $k \neq i, i+1$, $C_i \cap C_k = \emptyset$, does not exist in $R$.
\end{lemma}

\begin{proof}
Assume that such a polygon exists. We will translate this geometric statement into a purely algebraic statement, and then try to deduce a contradiction. Let $P_i$ be the generic point of $C_i$. Since $C_i \cap C_{i+1} = V(P_i + P_{i+1})$ is irreducible by hypothesis, it follows that $P_i + P_{i+1} \neq R$, hence prime \ref{irreducible-intersection-property-quadratically-closed-rings}. Also, $C_i \cap C_{k} = \emptyset$ for all $k \neq i, i+1$ (mod $n$) implies that $P_i + P_k = R$ for all $k \neq i, i+1$ (mod $n$). Hence, for all distinct $i, j, k \in \{0, \cdots, n-1\}$, $P_i + P_j + P_k = R$. Let $\alpha_{i,j} = P_i + P_j$ for all $i, j \in \{0, \cdots, n-1\}$ such that $i < j$. It follows that 
\begin{equation}
\prod_{i < j} \alpha_{i,j} = \prod_{0 \leq j \leq n-1} \alpha_{j, j+1}
\end{equation}
where the indices are mod $n$.
We also have the following sequence of ring homomorphisms
\begin{equation}
R/P_0 \cap \cdots P_{n-1} \xrightarrow{\varphi} \prod_{0 \leq i \leq n-1}R/P_i \xrightarrow{\phi} \prod_{0 \leq j \leq n-1} R/\alpha_{j,j+1} = \prod_{i < j} R/\alpha_{i,j}
\end{equation}
where $\varphi$ is the canonical injection, and $\phi$ is the map from \ref{generalized-exact-sequence}.

\medskip \noindent
Since $\alpha_{0,1} + \alpha_{1,2} = R$, by the Chinese Remainder Theorem, the canonical map $R \rightarrow R/\alpha_{0,1} \times R/\alpha_{1,2}$ is surjective, and $ker(R \rightarrow R/\alpha_{0,1} \times R/\alpha_{1,2}) = \alpha_{0,1} \cap \alpha_{1,2} \supset P_1$. Hence, we get a surjective ring map $R/P_1 \rightarrow R/\alpha_{0,1} \times R/\alpha_{1,2}$. In particular, there exists $\overline{c} \in R/P_1$, for $c \in R$, such that $(R/P_1 \rightarrow R/\alpha_{0,1} \times R/\alpha_{1,2})(\overline{c}) = (1, 0)$. Then consider the element $(0, \overline{c^2 - c}, 0, \cdots, 0) \in \prod_{0 \leq i \leq n-1}R/P_i$. The image of this element under $\phi$ is $0$. Thus, by \ref{exact-sequence-for-proof-of-quadratically-closed}, there exists $\beta \in R/P_0 \cap \cdots \cap P_{n-1}$ such that $\varphi(\beta) = (0, \overline{c^2 - c}, 0, \cdots, 0)$.

\medskip \noindent
Now consider the polynomial $x^2 - x - \beta \in R/P_0 \cap \cdots \cap P_{n-1}[x]$. Since $R/P_0 \cap \cdots \cap P_{n-1}$ is quadratically closed \ref{quadratically-closed-rings-and-quotients}, there exists $\gamma \in R/P_0 \cap \cdots \cap P_{n-1}$ such that $\gamma^2 - \gamma - \beta = 0$. Then $\varphi(\gamma)$ is a root of $x^2 - x - \varphi(\beta) = x^2 - x - (0,c^2 - c, 0, \cdots, 0)$. Since each $R/P_i$ is a domain, if $\varphi(\gamma) = (\overline{a_0}, \overline{a_1}, \cdots, \overline{a_{n-1}})$, for $a_j \in R$, then $\overline{a_j} = 0$ or $1$ for all $j \neq 1$, and $\overline{a_1} = \overline{c}$ or $\overline{1-c}$.

\medskip \noindent
Suppose $\overline{a_1} = \overline{c}$. Since $\phi(\varphi(\gamma)) = (\overline{\overline{a_0 - c}}, \overline{\overline{c - a_2}}, \cdots, \overline{\overline{a_{n-1} - a_0}}) = 0$, and $(R/P_1 \rightarrow R/\alpha_{0,1} \times R/\alpha_{1,2})(\overline{c}) = (1, 0)$, we get $\overline{a_0} = 1 \Rightarrow \overline{a_{n-1}} = 1 \Rightarrow \cdots \Rightarrow \overline{a_2} = 1$. Thus, $\varphi(\gamma) = (1, \overline{c}, 1, \cdots, 1)$. But then, $\overline{\overline{c- a_2}} = \overline{\overline{c - 1}} = -1 \neq 0$ in $R/\alpha_{1,2}$ (recall that $\alpha_{1,2}$ is prime, so $R/\alpha_{1,2}$ is not the zero ring), which is a contradiction. If on the other hand $\overline{a_1} = \overline{1 - c}$, then one can easily verify by a similar reasoning as above that $\varphi(\gamma) = (\overline{a_0}, \overline{a_1}, \cdots, \overline{a_{n-1}}) = (0, \overline{1 - c}, 0, \cdots, 0)$. Again, $\overline{\overline{a_1 - a_2}} = \overline{\overline{(1-c) - 0}} = 1 \neq 0$, which contradicts the fact that $\phi(\varphi(\gamma)) = 0$.
\end{proof}

\noindent
Since absolutely integrally closed rings are quadratically closed, we obtain the following as an immediate corollary:

\begin{lemma} 
\label{polygons-absolutely-integrally-closed-ring}
There are no polygons in an absolutely integrally closed ring.
\end{lemma}

\begin{lemma}
\label{polygons-2n-adically-closed-rings}
There are no polygons in a $2n$-adically closed ring.
\end{lemma}

\begin{proof}
It is clear that \ref{polygons-in-quadratically-closed-rings} and \ref{comment} are true for quadratically closed rings replaced by $2n$-adically closed rings. Hence, we can reduce to the situation in \ref{special-case-polygons-quadratically-closed}. The proof is then the same as in \ref{special-case-polygons-quadratically-closed}, except that now we choose our polynomial to be $(x^2 - x - \beta)^n$. If $\gamma$ is a root of $(x^2 - x - \beta)^n$ (such a root is guaranteed to exist, because $R/P_0 \cap \cdots P_{n-1}$ is $2n$-adically closed), then $(\varphi(\gamma)^2 - \varphi(\gamma) - \varphi(\beta))^n = 0$ in $\prod_{0\leq i \leq n-1} R/P_i$, and since the latter ring is reduced, we get $\varphi(\gamma)^2 - \varphi(\gamma) - \varphi(\beta) = \varphi(\gamma)^2 - \varphi(\gamma) - (0, \overline{c^2-c}, 0, \cdots, 0) = 0$. From there on, the proof is verbatim that of \ref{special-case-polygons-quadratically-closed}.
\end{proof}

\section{Acknowledgements}
\label{ack}
\noindent
I would like to thank Prof. Aise Johan de Jong, for his supervision of this paper. He introduced me to the fascinating world of schemes, and guided my studies in algebraic geometry. This thesis would not have been possible without his suggestions, and his ideas are present throughout this paper. He was always very jovial and approachable, and was never annoyed by my many questions. It was he who suggested that I read \cite{Art} and \cite{Hoch-Hun}, and thanks to this suggestion, I have had the opportunity to learn about constructions in commutative algebra, that might have otherwise escaped my attention. I would also like to mention how much more accessible commutative algebra and algebraic geometry have become, thanks to \cite{stacks-project} and \cite{foa}. Last, but not least, I would also like to thank every professor in the Mathematics Department of Columbia University, who taught and guided me throughout my undergraduate career, and Mary Young, who made life easier for me in this department.

\bibliographystyle{amsplain.bst}
\bibliography{Bibli(1)}

\end{document}